\documentclass[10pt,reqno]{amsart}

\usepackage[usenames, dvipsnames]{color}
\usepackage{latexsym}
\usepackage{amssymb}
\usepackage{amsfonts}
\usepackage{amsmath,amsthm,amscd}
\usepackage{graphicx}
\usepackage{subfigure}
\usepackage[sorted-cites,abbrev]{amsrefs}
\usepackage{MnSymbol}
\usepackage{mathrsfs}
\usepackage[english]{babel}
\usepackage[autostyle, english = american]{csquotes}
\usepackage{tikz}
\MakeOuterQuote{"}

\newcommand{\R}{{\mathbb R}}
\newcommand{\C}{{\mathbb C}}
\newcommand{\E}{{\mathbb E}}
\newcommand{\CP}{{\mathbb {CP}}}
\newcommand{\Z}{{\mathbb Z}}

\newcommand{\abs}[1]{|#1|}
\newcommand{\Abs}[1]{\big|#1\big|}
\newcommand{\norm}[1]{\|#1\|}
\newcommand{\half}{{\textstyle \frac{1}{2}}}
\newcommand{\vol}{{\operatorname{Vol}}}

\newcommand{\FS}{{{\operatorname{FS}}}}

\newcommand{\ccal}{\mathcal{C}}

\newcommand{\ecal}{\mathcal{E}}
\newcommand{\fcal}{\mathcal{F}}

\newcommand{\ocal}{\mathcal{O}}
\newcommand{\pcal}{\mathcal{P}}

\newcommand{\scal}{\mathcal{S}}

\newcommand{\al}{\alpha}
\newcommand{\be}{\beta}

\newcommand{\de}{\delta}
\newcommand{\De}{\Delta}
\newcommand{\ga}{\gamma}
\newcommand{\Ga}{\Gamma}

\newcommand{\La}{\Lambda}
\newcommand{\om}{\omega}
\newcommand{\Om}{\Omega}
\newcommand{\si}{\sigma}
\newcommand{\Si}{\Sigma}
\newcommand{\pa}{\partial}
\newcommand{\di}{\displaystyle}

\newtheorem{theorem}{Theorem}[section]
\newtheorem*{thm}{Theorem}
\newtheorem{cor}[theorem]{Corollary}
\newtheorem{lem}[theorem]{Lemma}

\newtheorem{rmk}[theorem]{Remark}

\newtheorem{definition}[theorem]{Definition}

\numberwithin{equation}{section}

\title[Hole probabilities of $SU(m+1)$ Gaussian random polynomials]{Hole probabilities of $SU(m+1)$ Gaussian random polynomials}

\author{Junyan Zhu}
\address{Department of Mathematics, Johns Hopkins University, Baltimore, MD 21218, USA}
\email{jyzhu@math.jhu.edu}
\thanks{}
\date{}

\begin{document}

\begin{abstract}
In this paper, we study hole probabilities $P_{0,m}(r,N)$ of $SU(m+1)$ Gaussian random polynomials of degree $N$ over a polydisc $(D(0,r))^m$. When $r\geq1$, we find asymptotic formulas and decay rate of $\log{P_{0,m}(r,N)}$. In dimension one, we also consider hole probabilities over some general open sets and compute asymptotic formulas for the generalized hole probabilities $P_{k,1}(r,N)$ over a disc $D(0,r)$.
\end{abstract}

\maketitle

\section*{0. Introduction}

Hole probability is the probability that some random field never vanishes over some set.
The case of Gaussian random entire functions was studied by Sodin and Tsirelson:
\begin{thm}[Sodin, Tsirelson\cite{ST}\ Theorem 1]
Let $\psi(z)=\di\sum_{k=0}^\infty c_k\frac{z^k}{\sqrt{k!}}$, where $c_k(k\geq0)$ are i.i.d. standard complex Gaussian random variables. Then $\exists\ C_1\geq C_2>0$ such that
\begin{align*}
\exp{\{-C_1r^4\}}\leq Prob\Big\{0\not\in\psi\big(D(0,r)\big)\Big\}\leq\exp{\{-C_2r^4\}}.
\end{align*}
\end{thm}
In \cite{SZZ}, the authors considered the case of Gaussian random sections: let $M$ be a compact K\"ahler manifold with complex dimension $m$ and $(L,h)\rightarrow M$ be a positive holomorphic line bundle. $\ga_N$ denotes the Gaussian probability measure on $H^0(M,L^N)$ induced by the fiberwised inner product $h^N$ and the polarized volume form $dV_M=\frac{\om_h^m}{m!}=\frac{1}{m!}(\frac{\sqrt{-1}}{2\pi}\Theta_h)^m$, where $\Theta_h$ is the Chern curvature tensor of $(L,h)$.
\begin{thm}[Shiffman, Zelditch, Zrebiec\cite{SZZ}\ Theorem 1.4]
For any nonempty open set $U\subset M$, if there exists $s\in H^0(M,L)$ such that $s$ does not vanish on $\bar U$. Then $\exists\ C_1\geq C_2>0$ such that for $N\gg1$,
\begin{align*}
\exp{\{-C_1N^{m+1}\}}\leq\ga_N\{s_N\in H^0(M,L^N):0\not\in s_N(U)\}\leq\exp{\{-C_2N^{m+1}\}}.
\end{align*}
\end{thm}
Therefore, it is natural to ask: can we find sharp constants $C_1$, $C_2$ in the above two theorems and furthermore, is it possible to obtain an asymptotic formula and a decay rate for the hole probability? Using Cauchy's integral estimates, Nishry answered this question in the random entire function case:
\begin{thm}[Nishry\cite{N1}\ Theorem 1]
Let $\psi(z)=\di\sum_{k=0}^\infty c_k\frac{z^k}{\sqrt{k!}}$, where $c_k(k\geq0)$ are i.i.d. standard complex Gaussian random variables. Then
\begin{align*}
Prob\Big\{0\not\in\psi\big(D(0,r)\big)\Big\}=\exp{\{-\frac{e^2}{2}r^4+O(r^{\frac{18}{5}})\}}.
\end{align*}
\end{thm}
This inspires us that for those line bundles with polynomial sections, maybe it is possible to find an asymptotic formula for the hole probability. \\

If $P_{0,m}(r,N)$ denotes the hole probability of $SU(m+1)$ Gaussian random polynomials over the polydisc $\big(D(0,r)\big)^m$, $d_mx$ is the Lebesgue measure on $\R^m$ and
\begin{align*}
E_r(x):=2\di\sum_{i=1}^{m}x_i\log{r}-\big[\di\sum_{i=1}^{m}x_i\log{x_i}+(1-\di\sum_{i=1}^{m}x_i)\log{(1-\di\sum_{i=1}^{m}x_i)}\big]
\end{align*}
is a continuous function defined over the standard simplex $\Si_m:=\{x=(x_1,\dots,x_m)\in\R^{m+}:\di\sum_{i=1}^{m}x_i\leq1\}$(here we adopt the convention that $0\log0=0$),
we have the following results:

\begin{theorem}\label{main2}
For $r\geq1$,
\begin{align*}
\log{P_{0,m}(r,N)}=-N^{m+1}\int_{\Si_m}E_r(x)\ d_mx+o(N^{m+1}),
\end{align*}
where
\begin{align*}
\int_{\Si_m}E_r(x)\ d_mx=\frac{2m\log{r}}{(m+1)!}+\frac{1}{m!}\di\sum_{k=2}^{m+1}\frac{1}{k}.
\end{align*}
\end{theorem}

\begin{theorem}\label{main1}
For $r>0$,
\begin{align*}
\log{P_{0,m}(r,N)}\geq-N^{m+1}\int_{x\in\Si_m:\ E_r(x)\geq0}E_r(x)\ d_mx+o(N^{m+1}), \\
\log{P_{0,m}(r,N)}\leq-N^{m+1}\int_{x\in\R^{m+}:\ \sum_{i=1}^{m}x_i\leq\al_0}E_r(x)\ d_mx+o(N^{m+1}),
\end{align*}
where
\begin{align*}
\al_0=\al_0(r,m)=
\begin{cases}
1&\text{ if }2\log{r}+\di\sum_{k=2}^{m}\frac{1}{k}\geq0, \\
\text{ the nonzero root of }(2\log{r}+\di\sum_{k=2}^{m}\frac{1}{k})\al=\al\log{\al}+(1-\al)\log{(1-\al)}&\text{ if }2\log{r}+\di\sum_{k=2}^{m}\frac{1}{k}<0.
\end{cases}
\end{align*}
Here when $m=1$, we take $\di\sum_{k=2}^{m}\frac{1}{k}=0$.
\end{theorem}

\begin{rmk}
 Theorem \ref{main2} can be derived from Theorem \ref{main1} as when $r\geq1$, $\{x\in\Si_m:\ E_r(x)\geq0\}=\Si_m$ and $\al_0(r,m)=1$. In fact we could have proved this general case directly. But the idea of the proof would turn out to be extremely difficult to follow.
\end{rmk}

\begin{cor}\label{main3}
In the case of $m=1$, the asymptotic formula for the logarithm of the hole probability over a disc exists for all $r>0$:
\begin{align*}
\log{P_{0,1}(r,N)}=-N^2\int_{0}^{\al_0}E_r(x)\ dx+o(N^2),
\end{align*}
here
\begin{align*}
\int_{0}^{\al_0}E_r(x)\ dx=\half\al_0(2\log{r}+1-\log{\al_0}),
\end{align*}
and $\al_0=\al_0(r,1)\in(0,1]$ is given in Theorem \ref{main1}.
\end{cor}

Because of the simplicity of one dimensional case, we can obtain more about the hole probability of $SU(2)$ Gaussian random polynomials:
\begin{theorem}\label{main4}
If $U\subset\C$ is a bounded simply connected domain containing $0$ and $\partial U$ is a Jordan curve. Let $\phi: D(0,1)\to U$ be a biholomorphism given by the Riemmann mapping theorem such that $\phi(0)=0$(thus $\phi$ is unique up to the composition of a unitary transformation of $\C$). Then the hole probability $P_{0,1}(U,N)$ of $SU(2)$ Gaussian random polynomials of degree $N$ over $U$ satisfies
\begin{align*}
\log{P_{0,1}(U,N)}\leq-(\log{\abs{\phi'(0)}}+\half)N^2+o(N^2).
\end{align*}
\end{theorem}

Also in dimension one, it makes sense to study the number of zeros in some set. So let a generalized hole probability $P_{k,1}(r,N)$ be the probability that an $SU(2)$ Gaussian random polynomial of degree $N$ has no more than $k$ zeros in $D(0,r)$, then the following theorem shows that asymptotic formula of $\log{P_{k,1}(r,N)}$ exists:
\begin{theorem}\label{main5}
For all $k\geq0$ and $r>0$:
\begin{align*}
\log{P_{k,1}(r,N)}=-\half\al_0(2\log{r}+1-\log{\al_0})N^2+o(N^2),
\end{align*}
where $\al_0=\al_0(r,1)\in(0,1]$ is given in Theorem \ref{main1}.
\end{theorem}

We should remark here that in all the cases we consider, the event that some Gaussian random polynomial has zeros on the boundary of some open set is a null set, i.e. of zero probability. Therefore we do not distinguish between the (generalized) hole probability over an open set and that over its closure.

\subsection*{Acknowledgement} The author would like to express his sincere gratitude to Professor Bernard Shiffman for his patience, motivation, enthusiasm and enlightening discussions. Without his guidance and persistent help this paper would not have been possible.

\section{Background}

We review in this section some background on $SU(m+1)$ Gaussian random polynomials and the definition of our probability measures. Before that, let's define two lexicographically ordered sets that will be consistently used as index sets throughout this paper.
\begin{definition}
\begin{align*}
\begin{split}
\Ga_{m,N}&:=\{J=(j_1,\dots,j_m)\in[0,N]^m\cap\Z^m:\ 0\leq j_1\leq\dots\leq j_m\leq N\}, \\
\La_{m,N}&:=\{K=(k_1,\dots,k_m)\in[0,N]^m\cap\Z^m:\ \abs{K}=k_1+\dots+k_m\leq N\}.
\end{split}
\end{align*}
\end{definition}
It is not difficult to show that $\abs{\Ga_{m,N}}=\abs{\La_{m,N}}=\binom{N+m}{m}$. \\

The tautological line bundle $\ocal(-1)$ over the complex projective space $\CP^m$ is a holomorphic line bundle with fibers
\begin{align*}
\ocal(-1)_{[x]}=\C\cdot x,\ \forall\ [x]=[x_0:\dots:x_m]\in\CP^m.
\end{align*}
Its dual bundle, denoted by $\ocal(1)$, is called the hyperplane section bundle since $\ocal(1)=[H]$ where the divisor
\begin{align*}
H=\{[x]\in\CP^m:\ x_0=0\}
\end{align*}
is a hyperplane in $\CP^m$. See \cite{GH} for details. By Theorem 15.5 in Chapter V of \cite{D}, $H^0(\CP^m,\ocal(N))$, the space of holomorphic sections of the tensor bundle $\ocal(N)=\ocal(1)^{\otimes N}$, is isomorphic to ${}^h\pcal^{N}_{m+1}$, the space of $(m+1)-$variable homogenous polynomials of degree $N$. The Fubini-Study metric $h_{\FS}$ on $\ocal(1)$ can be described in the following way: over the open subset
\begin{align*}
U_0=\big\{[x]=[x_0:\dots:x_m]\in\CP^m:\ x_0\neq0\big\}\subset\CP^m,
\end{align*}
we have a local frame of $\ocal(1)$
\begin{align*}
e([x])=x_0.
\end{align*}
Set
\begin{align*}
\norm{e([x])}^2_{h_{\FS}}=\frac{\abs{x_0}^2}{\di\sum_{i=0}^{m}\abs{x_i}^2}=\frac{\abs{x_0}^2}{\norm{x}^2},
\end{align*}
which is independent of the choice of representative $x$ of $[x]$.
In terms of affine coordinate
\begin{align*}
z=(z_1,\dots,z_m)=\Big(\frac{x_1}{x_0},\dots,\frac{x_m}{x_0}\Big)
\end{align*}
over $U_0$,
\begin{align*}
\norm{e(z)}^2_{h_{\FS}}=(1+\norm{z}^2)^{-1}=\big(1+\di\sum_{i=1}^{m}\abs{z_i}^2\big)^{-1},
\end{align*}
which defines a metric with positive Chern curvature form
\begin{align*}
\omega_{\FS}=-\frac{\sqrt{-1}}{2\pi}\pa\bar{\pa}\log{\norm{e(z)}^2_{h_{\FS}}}=\frac{\sqrt{-1}}{2\pi}\pa\bar{\pa}\log{(1+\abs{z_1}^2+\dots+\abs{z_m}^2)}.
\end{align*}
This induces a metric $h_{\FS}^N$ on the line bundle $\ocal(N)$ so that
\begin{align*}
\norm{e^{\otimes N}(z)}^2_{h^N_{\FS}}=(1+\norm{z}^2)^{-N}.
\end{align*}
With the frame $e^{\otimes N}$ over $U_0$, for any $s\in H^0(\CP^m,\ocal(N))$ which is represented by $p(x_0,\dots,x_m)\in{}^h\pcal^{N}_{m+1}$, we have
\begin{align*}
p(x_0,\dots,x_m)=\frac{p(x_0,\dots,x_m)}{x_0^N}e^{\otimes N}([x])=p(1,z_1,\dots,z_m)e^{\otimes N}([x]),
\end{align*}
which implies that all the elements in $H^0(\CP^m,\ocal(N))$ can be viewed over $U_0$ as polynomials in $(z_1,\dots,z_m)$ of degree at most $N$.

Since $\omega_{FS}$ is positive over $\CP^m$, we may take it as a polarized metric form on $\CP^m$ and the associated volume form is $dV=\frac{\omega_{\FS}^m}{m!}$. Thus, the metric $h_{\FS}^N$ together with the volume form $dV$ induce a Hermitian inner product on the space of holomorphic sections $H^0(\CP^m,\ocal(N))$: $\forall\ s_1,s_2\in H^0(\CP^m,\ocal(N))$,
\begin{align*}
\llangle s_1,s_2\rrangle:=\int_{\CP^m}\langle s_1,s_2\rangle_{h_{\FS}^N}\ dV.
\end{align*}
With this inner product, there is an orthonormal basis $\{S^N_K\}_{K=(k_1,\dots,k_m)\in\La_{m,N}}$, given in local affine coordinates $(z_1,\dots,z_m)$ over $U_0$ by
\begin{align*}
S^N_K(z)=\sqrt{(N+1)\cdots(N+m)}\sqrt{\binom{N}{K}}z^K,
\end{align*}
where we adopt the notations
\begin{align*}
\binom{N}{K}=\frac{N!}{(N-\abs{K})!k_1!\cdots k_m!},\ z^K:=z_1^{k_1}\cdots z_m^{k_m}.
\end{align*}

Thus $H^0(\CP^m,\ocal(N))=\big\{s_N=\di\sum_{K\in\La_{m,N}}c_KS^N_K:\ c=(c_K)_{K\in\La_{m,N}}\in\C^{\binom{N+m}{m}}\big\}$.
Endow $H^0(\CP^m,\ocal(N))$ with the Gaussian probability measure $\gamma_N$ defined by
\begin{align*}
d\gamma_N(s_N):=\pi^{-\binom{N+m}{m}}e^{-\norm{c}^2}\ d_{2\binom{N+m}{m}}c,
\end{align*}
where $\norm{c}^2=\di\sum_{K\in\La_{m,N}}\abs{c_K}^2$ and $d_{2\binom{N+m}{m}}c$ denotes the $2\binom{N+m}{m}-$dimensional Lebesgue measure. $\gamma_N$ is characterized by the property that $\{c_K\}_{K\in\La_{m,N}}$ are independent and identically distributed(i.i.d.) standard complex Gaussian random variables. Then $(H^0(\CP^m,\ocal(N)),\ \gamma_N)$ is called the ensemble of $SU(m+1)$ Gaussian random polynomials of degree $N$ as the random element $s_N$ is distributional invariant under $SU(m+1)$ transformations of $\CP^m$. Its hole probability over the polydisc $(D(0,r))^m\subset\C^m$ is
\begin{align*}
\begin{split}
P_{0,m}(r,N)&=\gamma_N\big\{s_N\in H^0(\CP^m,\ocal(N)):\ 0\not\in s_N\left((\bar{D}(0,r))^m\right)\big\} \\
            &=\pi^{-\binom{N+m}{m}}\int_{c\in\C^{\binom{N+m}{m}}:\ 0\not\in s_N\left((\bar{D}(0,r))^m\right)}e^{-\norm{c}^2}\ d_{2\binom{N+m}{m}}c \\
            &=\pi^{-\binom{N+m}{m}}\int_{c\in\C^{\binom{N+m}{m}}:\ 0\not\in\tilde{s}_N\left((\bar{D}(0,r))^m\right)}e^{-\norm{c}^2}\ d_{2\binom{N+m}{m}}c,
\end{split}
\end{align*}
where $\tilde{s}_N(z)=\di\sum_{K\in\La_{m,N}}c_K\sqrt{\binom{N}{K}}z^K$. Thereafter, when considering hole probability, we work on $\tilde{s}_N$ instead of $s_N$ for simplicity. \\

\section{Preliminaries}

\begin{definition}
$Q_{r,m}(N):=\di\sum_{K\in\La_{m,N}}\log\big[\binom{N}{K}r^{2\abs{K}}\big]$
\end{definition}

\begin{lem}\label{lem(Q)}
\begin{align*}
Q_{r,m}(N)=N^{m+1}\int_{\Si_m}E_r(x)\ d_mx+o(N^{m+1})=\Big[\frac{2m\log{r}}{(m+1)!}+\frac{1}{m!}\di\sum_{k=2}^{m+1}\frac{1}{k}\Big]N^{m+1}+o(N^{m+1}).
\end{align*}
\end{lem}
\begin{proof}
We can prove inductively that for $k\geq1$,
\begin{align}\label{lem(Q)ineq1}
\Big(\frac{k}{e}\Big)^k\leq k!\leq\frac{k^{k+1}}{e^{k-1}}\ \Leftrightarrow\ k\log{k}-k\leq\log{k!}\leq(k+1)\log{k}-(k-1).
\end{align}
\begin{align}\label{lem(Q)ineq2}
\Rightarrow\ -(k+1)\log{N}+(k-1)\leq k\log{\frac{k}{N}}-\log{k!}\leq-k\log{N}+k\text{ for }0\leq k\leq N.
\end{align}
$\forall K=(k_1,\dots,k_m)\in\La_{m,N}$,
\begin{align*}
\log\big[\binom{N}{K}r^{2\abs{K}}\big]-NE_r(\frac{K}{N})=\log{N!}+\di\sum_{i=1}^{m}(k_i\log{\frac{k_i}{N}}-\log{k_i!})+\big[(N-\abs{K})\log{\frac{N-\abs{K}}{N}}-\log{(N-\abs{K})!}\big],
\end{align*}
Applying (\ref{lem(Q)ineq1}) and (\ref{lem(Q)ineq2}), we get
\begin{align*}
\begin{split}
\log\big[\binom{N}{K}r^{2\abs{K}}\big]-NE_r(\frac{K}{N})&\geq(N\log{N}-N)-(N+m+1)\log{N}+(N-m-1)=-(m+1)(\log{N}+1), \\
\log\big[\binom{N}{K}r^{2\abs{K}}\big]-NE_r(\frac{K}{N})&\leq[(N+1)\log{N}-(N-1)]-N\log{N}+N=\log{N}+1,
\end{split}
\end{align*}
\begin{align*}
\Rightarrow\ \Abs{\log\big[\binom{N}{K}r^{2\abs{K}}\big]-NE_r(\frac{K}{N})}\leq(m+1)(\log{N}+1),\ \forall\ K\in\La_{m,N},
\end{align*}
\begin{align}\label{lem(Q)ineq3}
\begin{split}
\Rightarrow\ \Abs{Q_{r,m}(N)-N\di\sum_{K\in\La_{m,N}}E_r(\frac{K}{N})}&\leq\di\sum_{K\in\La_{m,N}}\Abs{\log\big[\binom{N}{K}r^{2\abs{K}}\big]-NE_r(\frac{K}{N})} \\
&\leq(m+1)(\log{N}+1)\binom{N+m}{m}=o(N^{m+1}).
\end{split}
\end{align}
Take
\begin{align*}
\mathring{\La}_{m,N}:=\{K\in\La_{m,N}:\ k_i\geq1\text{ for }1\leq i\leq m\text{ and }\abs{K}\leq N-m-1\}\subset\La_{m,N}
\end{align*}
and
\begin{align*}
\mathring{\Si}_m(N):=\di\bigcup_{K\in\mathring{\La}_{m,N}}[\frac{k_1}{N},\frac{k_1+1}{N}]\times\cdots\times[\frac{k_m}{N},\frac{k_m+1}{N}]\subset\Si_m.
\end{align*}
Then
\begin{align*}
\begin{split}
\abs{\mathring{\La}_{m,N}}&=\binom{N-m-1}{m}, \\
\abs{\La_{m,N}\setminus\mathring{\La}_{m,N}}&=\binom{N+m}{m}-\binom{N-m-1}{m}=O(N^{m-1}), \\
\vol_{\R^m}(\Si_m\setminus\mathring{\Si}_m(N))&=\frac{1}{m!}-N^{-m}\binom{N-m-1}{m}=O(N^{-1}).
\end{split}
\end{align*}
Over $\Si_m$ we have
\begin{align*}
\abs{E_r}\leq2\abs{\log{r}}+\frac{m+1}{e}=O(1),
\end{align*}
so
\begin{align}\label{lem(Q)ineq4}
\Abs{N\di\sum_{K\in\La_{m,N}}E_r(\frac{K}{N})-N\di\sum_{K\in\mathring{\La}_{m,N}}E_r(\frac{K}{N})}
\leq N\abs{\La_{m,N}\setminus\mathring{\La}_{m,N}}\di\sup_{\Si_m}\abs{E_r}=O(N^m).
\end{align}
As
\begin{align*}
\di\sup_{\mathring{\Si}_m(N)}\norm{\nabla E_r}\leq O(\log{N}),
\end{align*}
\begin{align}\label{lem(Q)ineq5}
\begin{split}
&\Abs{N\di\sum_{K\in\mathring{\La}_{m,N}}E_r(\frac{K}{N})-N^{m+1}\int_{\mathring{\Si}_m(N)}E_r(x)\ d_mx} \\
\leq&N^{m+1}\di\sum_{K\in\mathring{\La}_{m,N}}\di\int_{[\frac{k_1}{N},\frac{k_1+1}{N}]\times\cdots\times[\frac{k_m}{N},\frac{k_m+1}{N}]}\Abs{E_r(\frac{K}{N})-E_r(x)}\ d_mx \\
\leq&N^{m+1}\binom{N-m-1}{m}N^{-m}O(\log{N})O(N^{-1}) \\
&=O(N^m\log{N}).
\end{split}
\end{align}
\begin{align}\label{lem(Q)ineq6}
\Abs{N^{m+1}\int_{\mathring{\Si}_m(N)}E_r(x)\ d_mx-N^{m+1}\int_{\Si_m}E_r(x)\ d_mx}\leq N^{m+1}\di\sup_{\Si_m}\abs{E_r}\vol_{\R^m}(\Si_m\setminus\mathring{\Si}_m(N))=O(N^m).
\end{align}
Combining (\ref{lem(Q)ineq3})$\sim$(\ref{lem(Q)ineq6}), we thus obtain
\begin{align*}
\begin{split}
Q_{r,m}(N)
&=N^{m+1}\int_{\Si_m}E_r(x)\ d_mx+o(N^{m+1}) \\
&=N^{m+1}\int_{\Si_m}
2\di\sum_{i=1}^{m}x_i\log{r}-\big[\di\sum_{i=1}^{m}x_i\log{x_i}+(1-\di\sum_{i=1}^{m}x_i)\log{(1-\di\sum_{i=1}^{m}x_i)}\big]\ d_mx+o(N^{m+1}) \\
&=N^{m+1}\Big[2m\log{r}\int_{\Si_m}x_1\ d_mx-(m+1)\int_{\Si_m}x_1\log{x_1}\ d_mx\Big]+o(N^{m+1}) \\
&=\Big[\frac{2m\log{r}}{(m+1)!}+\frac{1}{m!}\di\sum_{k=2}^{m+1}\frac{1}{k}\Big]N^{m+1}+o(N^{m+1}).
\end{split}
\end{align*}
\end{proof}

\begin{rmk}
The scaled lattice $\frac{1}{N}\La_{m,N}\subset\R^m$ will tend to $\Si_m$. Hence Lemma \ref{lem(Q)} is in fact converting a Riemann sum into a Riemann integral and estimating the error. Such procedures will appear several times in this paper.
\end{rmk}

\begin{rmk}
The function $E_r(x)$ in the above lemma can also be written as $E_r(x)=-b_{\{x\}}(z_r)+\log{(1+\norm{z_r}^2)}$, where $z_r=(r,\dots,r)\in\R^m$ and $b_{\{x\}}$ is the exponential decay rate of the expected mass density of random $L^2$ normalized polynomials with some prescribed Newton polytope(see Theorem 1.2 and (78) in \cite{SZ}).
\end{rmk}

Let $\xi=(\xi_1,\dots,\xi_m)$, where for $1\leq i \leq m$, $\xi_i=(\xi_{i,0},\dots,\xi_{i,N})$.
\begin{definition}
$W_{m,N}(\xi)$ is the $\binom{N+m}{m}\times\binom{N+m}{m}$ matrix with rows indexed by $\Ga_{m,N}$ and columns indexed by $\La_{m,N}$, such that $\forall\ J=(j_1,\dots,j_m)\in\Ga_{m,N},\ K=(k_1,\dots,k_m)\in\La_{m,N}$, the $(J,K)$-entry of $W_{m,N}(\xi)$ is $\xi_J^K=\xi_{1,j_1}^{k_1}\cdots\xi_{m,j_m}^{k_m}$.
\end{definition}
Next lemma gives the formula for a "Vandermonde type" determinant.
\begin{lem}\label{lem(detW)}
$\abs{\det{W_{m,N}(\xi)}}=\di\prod_{i=1}^{m}\di\prod_{0\leq j<k\leq N}\abs{\xi_{i,j}-\xi_{i,k}}^{\binom{j+i-1}{i-1}\binom{N-k+m-i}{m-i}}$.
\end{lem}
\begin{proof}
$\forall\ 1\leq i\leq m$ and $0\leq j<k\leq N$, the rows of $W_{m,N}(\xi)$ involving $\xi_{i,j}$ correspond to the set
\begin{align*}
\Ga_{m,N}^{i,j}=\{(j_1,\dots,j_m)\in\Ga_{m,N}:j_i=j\}
\end{align*}
while those rows involving $\xi_{i,k}$ correspond to the set
\begin{align}\label{Ga(m,N,i,k)}
\Ga_{m,N}^{i,k}=\{(j_1,\dots,j_m)\in\Ga_{m,N}:j_i=k\}.
\end{align}
Let
\begin{align*}
\begin{split}
\tilde{\Ga}_{m,N}^{i,j}&=\{(j_1,\dots,\hat{j_i},\dots,j_m)\in[0,N]^{m-1}\cap\Z^{m-1}:0\leq j_1\leq\dots\leq j_{i-1}\leq j\leq j_{i+1}\leq\dots\leq j_m\leq N\}, \\
\tilde{\Ga}_{m,N}^{i,k}&=\{(j_1,\dots,\hat{j_i},\dots,j_m)\in[0,N]^{m-1}\cap\Z^{m-1}:0\leq j_1\leq\dots\leq j_{i-1}\leq k\leq j_{i+1}\leq\dots\leq j_m\leq N\},
\end{split}
\end{align*}
then
\begin{align*}
\begin{split}
\abs{\Ga_{m,N}^{i,j}}&=\abs{\tilde{\Ga}_{m,N}^{i,j}}=\binom{j+i-1}{i-1}\binom{N-j+m-i}{m-i}, \\
\abs{\Ga_{m,N}^{i,k}}&=\abs{\tilde{\Ga}_{m,N}^{i,k}}=\binom{k+i-1}{i-1}\binom{N-k+m-i}{m-i}.
\end{split}
\end{align*}
Since $\forall\ 1\leq i\leq m$,
\begin{align*}
\Ga_{m,N}=\di\bigsqcup_{k=0}^{N}\Ga_{m,N}^{i,k},
\end{align*}
we thus have the equality
\begin{align}\label{eq1}
\sum_{k=0}^{N}\binom{k+i-1}{i-1}\binom{N-k+m-i}{m-i}=\binom{N+m}{m}.
\end{align}
\begin{align*}
\tilde{\Ga}_{m,N}^{i,j}\cap\tilde{\Ga}_{m,N}^{i,k}=\{(j_1,\dots,\hat{j_i},\dots,j_m)\in[0,N]^{m-1}\cap\Z^{m-1}:0\leq j_1\leq\dots\leq j_{i-1}\leq j<k\leq j_{i+1}\leq\dots\leq j_m\leq N\}
\end{align*}
and
\begin{align*}
\abs{\tilde{\Ga}_{m,N}^{i,j}\cap\tilde{\Ga}_{m,N}^{i,k}}=\binom{j+i-1}{i-1}\binom{N-k+m-i}{m-i},
\end{align*}
which means that there are $\binom{j+i-1}{i-1}\binom{N-k+m-i}{m-i}$ pairs of rows, within each pair the only difference between two rows is replacing $\xi_{i,j}$ by $\xi_{i,k}$. Therefore, $\forall\ 1\leq i\leq m$ and $\forall\ 0\leq j<k\leq N$,
\begin{align*}
(\xi_{i,j}-\xi_{i,k})^{\binom{j+i-1}{i-1}\binom{N-k+m-i}{m-i}}|\det{W_{m,N}(\xi)},
\end{align*}
\begin{align}\label{division1}
\Rightarrow G_{m,N}(\xi):=\di\prod_{i=1}^{m}\di\prod_{0\leq j<k\leq N}(\xi_{i,j}-\xi_{i,k})^{\binom{j+i-1}{i-1}\binom{N-k+m-i}{m-i}}|\det{W_{m,N}(\xi)}.
\end{align}
$\forall\ 1\leq i\leq m$,
\begin{align}\label{eq2}
\begin{split}
\deg_{\xi_i}G_{m,N}(\xi)
&=\sum_{0\leq j<k\leq N}\binom{j+i-1}{i-1}\binom{N-k+m-i}{m-i} \\
&=\sum_{k=1}^{N}\big[\sum_{j=0}^{k-1}\binom{j+i-1}{i-1}\big]\binom{N-k+m-i}{m-i} \\
&=\sum_{k=1}^{N}\binom{k-1+i}{i}\binom{N-k+m-i}{m-i} \\
&=\sum_{k-1=0}^{N-1}\binom{(k-1)+(i+1)-1}{(i+1)-1}\binom{(N-1)-(k-1)+(m+1)-(i+1)}{(m+1)-(i+1)} \\
&=\binom{(N-1)+(m+1)}{m+1} \\
&=\binom{N+m}{m+1},
\end{split}
\end{align}
where the second to the last equality is due to (\ref{eq1}). On the other hand, $\forall\ 1\leq i\leq m$ and $1\leq k\leq N$, the number of K's in $\La_{m,N}$ with $k_i=k$ is $\binom{N-k+m-1}{m-1}$, so
\begin{align*}
\begin{split}
\deg_{\xi_i}\det{W_{m,N}(\xi)}
&=\sum_{k=1}^{N}k\binom{N-k+m-1}{m-1} \\
&=\binom{N+m}{m+1},
\end{split}
\end{align*}
where the second equality is the special case $i=1$ in (\ref{eq2}).
\begin{align}\label{division2}
\Rightarrow \deg_{\xi_i}\det{W_{m,N}(\xi)}=\deg_{\xi_i}G_{m,N}(\xi),\ \forall\ 1\leq i\leq m.
\end{align}
\begin{align*}
(\ref{division1})\text{ and }(\ref{division2})\Rightarrow
\det{W_{m,N}(\xi)}=C_{m,N}G_{m,N}=C_{m,N}\di\prod_{i=1}^{m}\di\prod_{0\leq j<k\leq N}(\xi_{i,j}-\xi_{i,k})^{\binom{j+i-1}{i-1}\binom{N-k+m-i}{m-i}},
\end{align*}
where $C_{m,N}$ is a constant depending only on $m$ and $N$. Consider the monomial
\begin{align*}
g_{m,N}(\xi):=\prod_{i=1}^{m}\prod_{k=1}^{N}\xi_{i,k}^{\sum_{j=0}^{k-1}\binom{j+i-1}{i-1}\binom{N-k+m-i}{m-i}}
=\prod_{i=1}^{m}\prod_{k=1}^{N}\xi_{i,k}^{\binom{k+i-1}{i}\binom{N-k+m-i}{m-i}},
\end{align*}
then
\begin{align*}
G_{m,N}(\xi)=\pm g_{m,N}(\xi)+\dots
\end{align*}
In the appendix, we show that the coefficient of $g_{m,N}$ in the expansion of $\det{W_{m,N}(\xi)}$ equals $1$, and therefore $C_{m,N}=\pm1$.
\end{proof}

\section{Lower bound in Theorem \ref{main2}}

\begin{proof}[Proof of the lower bound in Theorem \ref{main2}]
\begin{align}\label{ineq1}
\abs{\tilde{s}_N(z)}\geq\abs{c_{(0,\dots,0)}}-\sum_{K\in\La_{m,N}\backslash\{(0,\dots,0)\}}\abs{c_K}\sqrt{\binom{N}{K}}r^{\abs{K}},\ \forall\ z=(z_1,\dots,z_m)\in(\bar{D}(0,r))^m.
\end{align}
Consider the event $\Om_{r,m,N}$:
\begin{align*}
\begin{split}
&(i)\ \abs{c_{(0,\dots,0)}}\geq\sqrt{N}, \\
&(ii)\ \abs{c_K}\leq\frac{1}{2\sqrt{N}\sqrt{\binom{N}{K}}r^{\abs{K}}\binom{\abs{K}+m-1}{m-1}},\ K\in\La_{m,N}\backslash\{(0,\dots,0)\}.
\end{split}
\end{align*}
Then if $\Om_{r,m,N}$ occurs, by (\ref{ineq1}), we have $\forall\ z=(z_1,\dots,z_m)\in(\bar{D}(0,r))^m$,
\begin{align*}
\begin{split}
\abs{\tilde{s}_N(z)}
&\geq\sqrt{N}-\sum_{K\in\La_{m,N}\backslash\{(0,\dots,0)\}}\frac{\sqrt{\binom{N}{K}}r^{\abs{K}}}{2\sqrt{N}\sqrt{\binom{N}{K}}r^{\abs{K}}\binom{\abs{K}+m-1}{m-1}} \\
&=\sqrt{N}-\sum_{K\in\La_{m,N}\backslash\{(0,\dots,0)\}}\frac{1}{2\sqrt{N}\binom{\abs{K}+m-1}{m-1}} \\
&=\sqrt{N}-\sum_{k=1}^{N}\frac{1}{2\sqrt{N}} \\
&=\half\sqrt{N}>0,
\end{split}
\end{align*}
\begin{align*}
\Rightarrow P_{0,m}(r,N)\geq \gamma_N(\Om_{r,m,N})=\gamma_N(\abs{c_{(0,\dots,0)}}\geq\sqrt{N})
\prod_{K\in\La_{m,N}\backslash\{(0,\dots,0)\}}\gamma_N\bigg(\abs{c_K}\leq\frac{1}{2\sqrt{N}\sqrt{\binom{N}{K}}r^{\abs{K}}\binom{\abs{K}+m-1}{m-1}}\bigg),
\end{align*}
where $\gamma_N(\abs{c_{(0,\dots,0)}}\geq\sqrt{N})=e^{-N}$. Since $r\geq1$, $\frac{1}{2\sqrt{N}\sqrt{\binom{N}{K}}r^{\abs{K}}\binom{\abs{K}+m-1}{m-1}}\leq1$ for $K\in\La_{m,N}\backslash\{(0,\dots,0)\}$,
\begin{align*}
\gamma_N\bigg(\abs{c_K}\leq\frac{1}{2\sqrt{N}\sqrt{\binom{N}{K}}r^{\abs{K}}\binom{\abs{K}+m-1}{m-1}}\bigg)
\geq\half\bigg[\frac{1}{2\sqrt{N}\sqrt{\binom{N}{K}}r^{\abs{K}}\binom{\abs{K}+m-1}{m-1}}\bigg]^2
=\frac{1}{8N\binom{N}{K}r^{2\abs{K}}{\binom{\abs{K}+m-1}{m-1}}^2},
\end{align*}
\begin{align*}
\log{P_{0,m}(r,N)}\geq
-N-\sum_{K\in\La_{m,N}\backslash\{(0,\dots,0)\}}\Big\{\log{8}+\log{N}+2\log{\binom{\abs{K}+m-1}{m-1}}+\log{\big[\binom{N}{K}r^{2\abs{K}}\big]}\Big\},
\end{align*}
where
\begin{align*}
\log{\binom{\abs{K}+m-1}{m-1}}\leq\log{\binom{N+m-1}{m-1}}=O(\log{N}),
\end{align*}
\begin{align*}
\Rightarrow \sum_{K\in\La_{m,N}\backslash\{(0,\dots,0)\}}\Big[\log{8}+\log{N}+2\log{\binom{\abs{K}+m-1}{m-1}}\Big]=\binom{N+m}{m}O(\log{N})=o(N^{m+1}),
\end{align*}
\begin{align*}
\begin{split}
\Rightarrow \log{P_{0,m}(r,N)}
&\geq-\sum_{K\in\La_{m,N}\backslash\{(0,\dots,0)\}}\log{\big[\binom{N}{K}r^{2\abs{K}}\big]}+o(N^{m+1}) \\
&=-Q_{r,m}(N)+o(N^{m+1})=-N^{m+1}\int_{\Si_m}E_r(x)\ d_mx+o(N^{m+1}).
\end{split}
\end{align*}
\end{proof}

\section{Upper bound in Theorem \ref{main2}}

Let $\de>0$ be small, $\kappa=1-\sqrt{\de}$. We shall first treat $\de$ as a small positive constant and at the end we will let $\de\to0+$. For the sake of clarity, all the constants $C$, capital $O$ and little $o$ terms listed throughout this paper will not depend on $\de$ unless stated.
\begin{definition}
$z_j(N):=\kappa re^{2\pi\sqrt{-1}\frac{j}{N+1}}$, for $0\leq j\leq N$.
\end{definition}

$\forall\ p\in\Z^+$, assume $N+1=q(N)p+l(N)$, where $q(N)\in\Z$, $q(N)\geq0$ and $0\leq l(N)<p$. For convenience, we drop the dependence of $N$ when there is no confusion. $\forall\ 1\leq i\leq m$, assign the values of $\xi_i=(\xi_{i,0},\dots,\xi_{i,N})$ by means of the table below:
\begin{align}\label{table}
\begin{array}{|l|l|l|l|}
\hline
\xi_{i,0}=z_0 & \cdots & \xi_{i,(q-1)p}=z_{q-1} & \xi_{i,qp}=z_q \\
\hline
\xi_{i,1}=z_{q+1} & \cdots & \xi_{i,(q-1)p+1}=z_{(q+1)+(q-1)} & \xi_{i,qp+1}=z_{(q+1)+q} \\
\hline
\cdots\cdots\cdots\cdots\cdots\cdots\cdots\cdots\cdots & \cdots & \cdots\cdots\cdots\cdots\cdots\cdots\cdots\cdots\cdots\cdots\cdots\cdots\cdots & \cdots\cdots\cdots\cdots\cdots\cdots\cdots\cdots\cdots \\
\hline
\xi_{i,l-1}=z_{(l-1)(q+1)} & \cdots & \xi_{i,(q-1)p+(l-1)}=z_{(l-1)(q+1)+(q-1)} & \xi_{i,qp+(l-1)}=z_{(l-1)(q+1)+q} \\
\hline
\xi_{i,l}=z_{l(q+1)} & \cdots & \xi_{i,(q-1)p+l}=z_{l(q+1)+(q-1)} & \\
\hline
\cdots\cdots\cdots\cdots\cdots\cdots\cdots\cdots\cdots & \cdots & \cdots\cdots\cdots\cdots\cdots\cdots\cdots\cdots\cdots\cdots\cdots\cdots\cdots & \\
\hline
\xi_{i,p-1}=z_{l(q+1)+(p-1-l)q} & \cdots & \xi_{i,(q-1)p+(p-1)}=z_{l(q+1)+(p-1-l)q+(q-1)} & \\
\hline
\end{array}
\end{align}


Intuitively, table (\ref{table}) gives a way to choose points $\xi_{i,j}(j=0,1,\dots)$ one after another on the circle of radius $\kappa r$ that the arguments of each two consecutive points differ approximately by $\frac{2\pi}{p}$. Denote the bijection of $N+1$ letters $\{0,\dots,N\}$ indicated in table (\ref{table}) by $\tau$, i.e. $z_j=\xi_{i,\tau(j)}$ for $0\leq j\leq N$ and $1\leq i\leq m$. Denote
\begin{align*}
\begin{split}
&I_0=\{0,\dots,q\},\ a_0=0, \\
&I_1=\{q+1,\dots,(q+1)+q\},\ a_1=q+1, \\
&\dots \\
&I_{l-1}=\{(l-1)(q+1),\dots,(l-1)(q+1)+q\},\ a_{l-1}=(l-1)(q+1), \\
&I_l=\{l(q+1),\dots,l(q+1)+(q-1)\},\ a_l=l(q+1), \\
&\dots \\
&I_{p-1}=\{l(q+1)+(p-1-l)q,\dots,l(q+1)+(p-1-l)q+(q-1)\},\ a_{p-1}=l(q+1)+(p-1-l)q.
\end{split}
\end{align*}
$I_0,\dots,I_{p-1}$ give a partition of $\{0,\dots,N\}$. Again there is an implicit dependence on $N$ for each term defined above, and we would show this dependence explicitly when necessary. Then
\begin{align*}
a_t=tq+\min\{t,l\}=
\begin{cases}
t(q+1) & \text{ when } j\in I_t,\ 0\leq t\leq l, \\
l(q+1)+(t-l)q & \text{ when } j\in I_t,\ l+1\leq t\leq p-1,
\end{cases}
\end{align*}
\begin{align*}
\tau(j)=(j-a_t)p+t=
\begin{cases}
[j-t(q+1)]p+t & \text{ when } j\in I_t,\ 0\leq t\leq l, \\
[j-l(q+1)-(t-l)q]p+t & \text{ when } j\in I_t,\ l+1\leq t\leq p-1,
\end{cases}
\end{align*}
and if $\{j(N)\}_{N=1}^{\infty}$ is a sequence satisfying $j(N)\in I_t(N)$, $\forall\ N\geq1$,
\begin{align*}
\Abs{\tau_N(j(N))-pj(N)+t(N+1)}\leq2p^2,
\end{align*}
\begin{align}\label{tau}
\Rightarrow\frac{\tau_N(j(N))}{N+1}-\big(p\frac{j(N)}{N+1}-t\big)=O(N^{-1}).
\end{align}
\begin{lem}\label{lem(detW1)}
With the assignment of the values of $\xi_i$ given in (\ref{table}),
\begin{align*}
\log\abs{\det{W_{m,N}(\xi)}}=m\binom{N+m}{m+1}\log{(\kappa r)}+\frac{\be_m}{p}N^{m+1}+o(N^{m+1}),
\end{align*}
where $\be_m=\frac{1}{(m-1)!}\di\int_0^1x^m\log[2\sin(\pi x)]\ dx$, which is finite for each $m\geq1$ by comparison test of improper integrals.
\end{lem}
\begin{proof}
By Lemma \ref{lem(detW)},
\begin{align*}
\begin{split}
\log\abs{\det{W_{m,N}(\xi)}}
=&\log\big[\di\prod_{i=1}^{m}\di\prod_{0\leq j<k\leq N}\abs{\xi_{i,j}-\xi_{i,k}}^{\binom{j+i-1}{i-1}\binom{N-k+m-i}{m-i}}\big] \\
=&\di\sum_{i=1}^{m}\di\sum_{0\leq j<k\leq N}\binom{j+i-1}{i-1}\binom{N-k+m-i}{m-i}\log{\Abs{\frac{\xi_{i,j}}{\kappa r}-\frac{\xi_{i,k}}{\kappa r}}} \\
&+\di\sum_{i=1}^{m}\di\sum_{0\leq j<k\leq N}\binom{j+i-1}{i-1}\binom{N-k+m-i}{m-i}\log(\kappa r) \\
=&\di\sum_{i=1}^{m}\di\sum_{0\leq \tau(j)<\tau(k)\leq N}\binom{\tau(j)+i-1}{i-1}\binom{N-\tau(k)+m-i}{m-i}\log{\Abs{\frac{\xi_{i,\tau(j)}}{\kappa r}-\frac{\xi_{i,\tau(k)}}{\kappa r}}} \\
&+m\binom{N+m}{m+1}\log(\kappa r) \\
=&\di\sum_{i=1}^{m}\di\sum_{0\leq \tau(j)<\tau(k)\leq N}
\binom{\tau(j)+i-1}{i-1}\binom{N-\tau(k)+m-i}{m-i}\log{\abs{e^{2\pi\sqrt{-1}\frac{j}{N+1}}-e^{2\pi\sqrt{-1}\frac{k}{N+1}}}} \\
&+m\binom{N+m}{m+1}\log(\kappa r)
\end{split}
\end{align*}
where the second part of the third equality is due to (\ref{eq2}). We are going to show that the sum in the last equality can be approximated by a double integral.
\begin{align}\label{detW0}
\begin{split}
&\di\sum_{i=1}^{m}\di\sum_{0\leq \tau(j)<\tau(k)\leq N}
\binom{\tau(j)+i-1}{i-1}\binom{N-\tau(k)+m-i}{m-i}\log{\abs{e^{2\pi\sqrt{-1}\frac{j}{N+1}}-e^{2\pi\sqrt{-1}\frac{k}{N+1}}}} \\
=&\di\sum_{i=1}^{m}\di\sum_{0\leq \tau(j)<\tau(k)\leq N}
\big[\frac{(\tau(j))^{i-1}}{(i-1)!}+o((\tau(j))^{i-1})\big]\big[\frac{(N-\tau(k))^{m-i}}{(m-i)!}+o((N-\tau(k))^{m-i})\big]
\log{\abs{1-e^{2\pi\sqrt{-1}(\frac{j}{N+1}-\frac{k}{N+1})}}},
\end{split}
\end{align}
$\forall\ 1\leq i\leq m, 0\leq u,v\leq p-1$, denote
\begin{align*}
L_{u,v,N}=\{(j,k)\in I_u\times I_v:\ \tau(j)<\tau(k)\},
\end{align*}
\begin{align*}
T_{u,v}(N)=\bigcup_{(j,k)\in L_{u,v,N}}[\frac{j}{N+1},\frac{j+1}{N+1}]\times[\frac{k}{N+1},\frac{k+1}{N+1}],
\end{align*}
\begin{align*}
\mathring{L}_{u,v,N}=\{(j,k)\in L_{u,v,N}:\ j-k\neq\pm N\text{ and }j-k\neq\pm1\}\subset L_{u,v,N},
\end{align*}
\begin{align*}
\mathring{T}_{u,v}(N)=\bigcup_{(j,k)\in \mathring{L}_{u,v,N}}[\frac{j}{N+1},\frac{j+1}{N+1}]\times[\frac{k}{N+1},\frac{k+1}{N+1}]\subset T_{u,v}(N),
\end{align*}
and a function defined over $\{(x,y)\in(0,1)\times(0,1):\ x\neq y\}$:
\begin{align*}
g^{i}_{u,v}(x,y)=(px-u)^{i-1}[1-(py-v)]^{m-i}\log{\abs{1-e^{2\pi\sqrt{-1}(x-y)}}}.
\end{align*}
Then
\begin{align}\label{detW1}
\Abs{L_{u,v,N}\setminus \mathring{L}_{u,v,N}}\leq2N+2,
\end{align}
\begin{align}\label{detW2}
\vol_{\R^2}(T_{u,v}(N)\setminus\mathring T_{u,v}(N))\leq O(N^{-1}),
\end{align}
\begin{align}\label{detW3}
\frac{1}{N+1}\leq\Abs{\frac{j-k}{N+1}}\leq\frac{N}{N+1}\text{ for }(j,k)\in L_{u,v,N},
\end{align}
\begin{align}\label{detW4}
\frac{1}{N+1}\leq\Abs{x-y}\leq\frac{N}{N+1}\text{ for }(x,y)\in \mathring{T}_{u,v}(N),
\end{align}
\begin{align}\label{detW5}
\abs{g^i_{u,v}(x,y)}\leq O(\log{N})\text{ if }\frac{1}{N+1}\leq\abs{x-y}\leq\frac{N}{N+1},
\end{align}
\begin{align}\label{detW6}
\norm{\nabla g^i_{u,v}(x,y)}\leq O(N^{\half})\text{ if }\frac{1}{\sqrt{N+1}}\leq\abs{x-y}\leq1-\frac{1}{\sqrt{N+1}}.
\end{align}
From (\ref{tau}), we have
\begin{align}\label{detW7}
\begin{split}
&\di\sum_{0\leq \tau(j)<\tau(k)\leq N}(\tau(j))^{i-1}(N-\tau(k))^{m-i}\log{\abs{1-e^{2\pi\sqrt{-1}(\frac{j}{N+1}-\frac{k}{N+1})}}} \\
=&(N+1)^{m-1} \\
&\times\di\sum_{0\leq u,v\leq p-1}\di\sum_{(j,k)\in L_{u,v,N}}[p\frac{j}{N+1}-u+O(N^{-1})]^{i-1}[1-(p\frac{k}{N+1}-v)+O(N^{-1})]^{m-i}\log{\abs{1-e^{2\pi\sqrt{-1}(\frac{j}{N+1}-\frac{k}{N+1})}}},
\end{split}
\end{align}
$\forall\ 0\leq u,v\leq p-1$, by (\ref{detW1}), (\ref{detW3}) and (\ref{detW5}),
\begin{align}\label{detW8}
\begin{split}
&\di\sum_{(j,k)\in L_{u,v,N}}(p\frac{j}{N+1}-u)^{i-1}[1-(p\frac{k}{N+1}-v)]^{m-i}\log{\abs{1-e^{2\pi\sqrt{-1}(\frac{j}{N+1}-\frac{k}{N+1})}}} \\
=&\di\sum_{(j,k)\in L_{u,v,N}}g^i_{u,v}(\frac{j}{N+1},\frac{k}{N+1}) \\
=&\di\sum_{(j,k)\in\mathring L_{u,v,N}}g^i_{u,v}(\frac{j}{N+1},\frac{k}{N+1})+O(N\log{N}),
\end{split}
\end{align}
\begin{align}\label{detW9}
\begin{split}
&\Abs{(N+1)^{-2}\di\sum_{(j,k)\in\mathring{L}_{u,v,N}}g^i_{u,v}(\frac{j}{N+1},\frac{k}{N+1})-\di\iint_{\mathring{T}_{u,v}(N)}g^i_{u,v}(x,y)\ dxdy} \\
\leq&\di\sum_{(j,k)\in\mathring{L}_{u,v,N}}\iint_{[\frac{j}{N+1},\frac{j+1}{N+1}]\times[\frac{k}{N+1},\frac{k+1}{N+1}]}\Abs{g^i_{u,v}(x,y)-g^i_{u,v}(\frac{j}{N+1},\frac{k}{N+1})}\ dxdy \\
=&\di\sum_{(j,k)\in\mathring{L}_{u,v,N}:\frac{1}{\sqrt{N+1}}\leq\abs{\frac{j-k}{N+1}}\leq1-\frac{1}{\sqrt{N+1}}}\iint_{[\frac{j}{N+1},\frac{j+1}{N+1}]\times[\frac{k}{N+1},\frac{k+1}{N+1}]}\Abs{g^i_{u,v}(x,y)-g^i_{u,v}(\frac{j}{N+1},\frac{k}{N+1})}\ dxdy \\
&+\di\sum_{(j,k)\in\mathring{L}_{u,v,N}:\abs{\frac{j-k}{N+1}}<\frac{1}{\sqrt{N+1}}\text{ or }\abs{\frac{j-k}{N+1}}>1-\frac{1}{\sqrt{N+1}}}\iint_{[\frac{j}{N+1},\frac{j+1}{N+1}]\times[\frac{k}{N+1},\frac{k+1}{N+1}]}\Abs{g^i_{u,v}(x,y)-g^i_{u,v}(\frac{j}{N+1},\frac{k}{N+1})}\ dxdy.
\end{split}
\end{align}
Since
\begin{align*}
\Abs{(j,k)\in\mathring{L}_{u,v,N}:\frac{1}{\sqrt{N+1}}\leq\abs{\frac{j-k}{N+1}}\leq1-\frac{1}{\sqrt{N+1}}}\leq\Abs{\mathring L_{u,v,N}}=O(N^2),
\end{align*}
\begin{align*}
\Abs{(j,k)\in\mathring{L}_{u,v,N}:\abs{\frac{j-k}{N+1}}<\frac{1}{\sqrt{N+1}}\text{ or }\abs{\frac{j-k}{N+1}}>1-\frac{1}{\sqrt{N+1}}}\leq O(N^{\frac{3}{2}}),
\end{align*}
\begin{align}\label{detW10}
\begin{split}
(\ref{detW6})\ \Rightarrow\ &\di\sum_{(j,k)\in\mathring{L}_{u,v,N}:\frac{1}{\sqrt{N+1}}\leq\abs{\frac{j-k}{N+1}}\leq1-\frac{1}{\sqrt{N+1}}}\iint_{[\frac{j}{N+1},\frac{j+1}{N+1}]\times[\frac{k}{N+1},\frac{k+1}{N+1}]}\Abs{g^i_{u,v}(x,y)-g^i_{u,v}(\frac{j}{N+1},\frac{k}{N+1})}\ dxdy \\
\leq&O(N^2)\times(N+1)^{-2}\times\frac{\sqrt{2}}{N+1}\times\di\sup_{\frac{1}{\sqrt{N+1}}\leq\abs{x-y}\leq1-\frac{1}{\sqrt{N+1}}}\norm{\nabla g^i_{u,v}(x,y)} \\
\leq&O(N^{-\half}),
\end{split}
\end{align}
and by (\ref{detW4}), (\ref{detW5}),
\begin{align}\label{detW11}
\begin{split}
&\di\sum_{(j,k)\in\mathring{L}_{u,v,N}:\abs{\frac{j-k}{N+1}}<\frac{1}{\sqrt{N+1}}\text{ or }\abs{\frac{j-k}{N+1}}>1-\frac{1}{\sqrt{N+1}}}\iint_{[\frac{j}{N+1},\frac{j+1}{N+1}]\times[\frac{k}{N+1},\frac{k+1}{N+1}]}\Abs{g^i_{u,v}(x,y)-g^i_{u,v}(\frac{j}{N+1},\frac{k}{N+1})}\ dxdy \\
\leq&O(N^{\frac{3}{2}})\times(N+1)^{-2}\times O(\log{N}) \\
=&O(N^{-\half}\log{N}).
\end{split}
\end{align}
Denote $T_{u,v}=\{(x,y)\in\R^2:\ 0\leq x-\frac{u}{p}\leq y-\frac{v}{p}\leq\frac{1}{p}\}$. Since $g^i_{u,v}$ is $L^1_{loc}$, the measure $g^i_{u,v}(x,y)\ dxdy$ is absolutely continuous with respect to the Lebesgue measure. Thus by lemma \ref{lem(detW2)} below, we have
\begin{align}\label{detW12}
\di\iint_{\mathring{T}_{u,v}(N)}g^i_{u,v}(x,y)\ dxdy-\di\iint_{T_{u,v}}g^i_{u,v}(x,y)\ dxdy=o(1)\text{ as }N\to\infty.
\end{align}
\begin{align}\label{detW13}
\begin{split}
(\ref{detW8})\sim(\ref{detW12})\ \Rightarrow\ &
\di\sum_{(j,k)\in L_{u,v,N}}(p\frac{j}{N+1}-u)^{i-1}[1-(p\frac{k}{N+1}-v)]^{m-i}\log{\abs{1-e^{2\pi\sqrt{-1}(\frac{j}{N+1}-\frac{k}{N+1})}}} \\
=&(N+1)^2\di\iint_{T_{u,v}}g^i_{u,v}(x,y)\ dxdy+o(N^2).
\end{split}
\end{align}
\begin{align}\label{detW14}
\begin{split}
(\ref{detW13})+(\ref{detW7})\ \Rightarrow\ &\di\sum_{0\leq \tau(j)<\tau(k)\leq N}(\tau(j))^{i-1}(N-\tau(k))^{m-i}\log{\abs{1-e^{2\pi\sqrt{-1}(\frac{j}{N+1}-\frac{k}{N+1})}}} \\
=&(N+1)^{m+1}\di\sum_{0\leq u,v\leq p-1}\iint_{T_{u,v}}g^i_{u,v}(x,y)\ dxdy+o(N^{m+1}),
\end{split}
\end{align}
\begin{align*}
\begin{split}
(\ref{detW14})+(\ref{detW0})\ \Rightarrow\ &\di\sum_{i=1}^{m}\di\sum_{0\leq \tau(j)<\tau(k)\leq N}
\binom{\tau(j)+i-1}{i-1}\binom{N-\tau(k)+m-i}{m-i}\log{\abs{e^{2\pi\sqrt{-1}\frac{j}{N+1}}-e^{2\pi\sqrt{-1}\frac{k}{N+1}}}} \\
=&\di\sum_{i=1}^m\di\sum_{0\leq u,v\leq p-1}\iint_{T_{u,v}}\frac{g^i_{u,v}(x,y)}{(i-1)!(m-i)!}\ dxdy+o(N^{m+1}) \\
=&\di\sum_{i=1}^{m}\di\sum_{0\leq u,v\leq p-1}
\di\iint_{T_{u,v}}\frac{[p(x-\frac{u}{p})]^{i-1}}{(i-1)!}\frac{[1-p(y-\frac{v}{p})]^{m-i}}{(m-i)!}\log{\abs{1-e^{2\pi\sqrt{-1}(x-y)}}}\ dxdy+o(N^{m+1}) \\
=&\di\sum_{i=1}^{m}\di\sum_{0\leq u,v\leq p-1}
\di\iint_{T_{0,0}}\frac{(px)^{i-1}}{(i-1)!}\frac{(1-py)^{m-i}}{(m-i)!}\log{\abs{1-e^{2\pi\sqrt{-1}(x-y+\frac{u}{p}-\frac{v}{p})}}}\ dxdy+o(N^{m+1}) \\
=&\di\sum_{i=1}^{m}\di\sum_{0\leq u\leq p-1}
\di\iint_{T_{0,0}}\frac{(px)^{i-1}}{(i-1)!}\frac{(1-py)^{m-i}}{(m-i)!}
\log\big[\di\prod_{v=0}^{p-1}{\abs{e^{2\pi\sqrt{-1}\frac{v}{p}}-e^{2\pi\sqrt{-1}(x-y+\frac{u}{p})}}}\big]\ dxdy+o(N^{m+1}) \\
=&p\di\sum_{i=1}^{m}\di\iint_{T_{0,0}}\frac{(px)^{i-1}}{(i-1)!}\frac{(1-py)^{m-i}}{(m-i)!}\log{\abs{1-e^{2\pi\sqrt{-1}(px-py)}}}\ dxdy+o(N^{m+1}) \\
=&\frac{1}{p}\di\iint_T\di\sum_{i=1}^{m}\frac{x^{i-1}}{(i-1)!}\frac{(1-y)^{m-i}}{(m-i)!}\log{\abs{1-e^{2\pi\sqrt{-1}(x-y)}}}\ dxdy+o(N^{m+1}) \\
=&\frac{1}{p(m-1)!}\di\iint_T(1+x-y)^{m-1}\log{\abs{1-e^{2\pi\sqrt{-1}(x-y)}}}\ dxdy+o(N^{m+1}),
\end{split}
\end{align*}
where $T=\{(x,y)\in\R^2:\ 0\leq x\leq y\leq1\}$. Make change of variables: $\tilde{x}=x-y,\ \tilde{y}=y$, then $T$ is mapped to $\tilde{T}=\{(\tilde{x},\tilde{y})\in\R^2:\ -1\leq\tilde{x}\leq0,\ -\tilde{x}\leq\tilde{y}\leq1\}$.
\begin{align*}
\begin{split}
&\frac{1}{(m-1)!}\di\iint_T(1+x-y)^{m-1}\log{\abs{1-e^{2\pi\sqrt{-1}(x-y)}}}\ dxdy \\
=&\frac{1}{(m-1)!}\di\iint_{\tilde{T}}(1+\tilde{x})^{m-1}\log{\abs{1-e^{2\pi\sqrt{-1}\tilde{x}}}}\ d\tilde{x}d\tilde{y} \\
=&\frac{1}{(m-1)!}\di\int_{-1}^{0}(1+\tilde{x})^m\log{\abs{1-e^{2\pi\sqrt{-1}\tilde{x}}}}\ d\tilde{x} \\
=&\frac{1}{(m-1)!}\di\int_0^1x^m\log{\abs{1-e^{2\pi\sqrt{-1}x}}}\ dx \\
=&\frac{1}{(m-1)!}\di\int_0^1x^m\log[2\sin(\pi x)]\ dx \\
=&\be_m,
\end{split}
\end{align*}
\begin{align*}
\Rightarrow\di\sum_{i=1}^{m}\di\sum_{0\leq \tau(j)<\tau(k)\leq N}
\binom{\tau(j)+i-1}{i-1}\binom{N-\tau(k)+m-i}{m-i}\log{\abs{e^{2\pi\sqrt{-1}\frac{j}{N+1}}-e^{2\pi\sqrt{-1}\frac{k}{N+1}}}}
=\frac{\be_m}{p}N^{m+1}+o(N^{m+1}),
\end{align*}
\begin{align*}
\Rightarrow \log\abs{\det{W_{m,N}(\xi)}}=m\binom{N+m}{m+1}\log{(\kappa r)}+\frac{\be_m}{p}N^{m+1}+o(N^{m+1}).
\end{align*}
\end{proof}
\begin{lem}\label{lem(detW2)}
$\di\lim_{N\to\infty}\vol_{\R^2}(T_{u,v}\triangle\mathring T_{u,v}(N))=0$ for any $0\leq u,v\leq p-1$.
\end{lem}
\begin{proof}
By (\ref{detW2}), it is equivalent to show that $\di\lim_{N\to\infty}\vol_{\R^2}(T_{u,v}\triangle T_{u,v}(N))=0$, which is a direct consequence of $\di\lim_{N\to\infty}T_{u,v}(N)\setminus\partial T_{u,v}=\mathring T_{u,v}$. \\

First let's show $\di\limsup_{N\to\infty}T_{u,v}(N)\subset T_{u,v}$. $\forall(x,y)\in\di\limsup_{N\to\infty}T_{u,v}(N)$, $\exists\ \{N_n\}_{n=1}^{\infty}\to\infty$ such that $\forall\ n\geq1$, $\exists\ (j(N_n),k(N_n))\in I_u(N_n)\times I_v(N_n)$, $\tau_{N_n}(j(N_n))<\tau_{N_n}(k(N_n))$ and $(x,y)\in\big[\frac{j(N_n)}{N_n+1},\frac{j(N_n)+1}{N_n+1}\big]\times\big[\frac{k(N_n)}{N_n+1},\frac{k(N_n)}{N_n+1}\big]$. Then $\di\lim_{n\to\infty}\frac{j(N_n)}{N_n+1}=x$, $\di\lim_{n\to\infty}\frac{k(N_n)}{N_n+1}=y$. Since $0\leq\frac{\tau_{N_n}(j(N_n))}{N_n+1}<\frac{\tau_{N_n}(k(N_n))}{N_n+1}\leq\frac{N_n}{N_n+1}$ and $(j(N_n),k(N_n))\in I_u(N_n)\times I_v(N_n)$, (\ref{tau}) implies that $0\leq p\di\lim_{n\to\infty}\frac{j(N_n)}{N_n+1}-u\leq p\di\lim_{n\to\infty}\frac{k(N_n)}{N_n+1}-v\leq1$. Hence $0\leq px-u\leq py-v\leq1$ and $(x,y)\in T_{u,v}$. \\

Next we will prove $\mathring T_{u,v}\subset\di\liminf_{N\to\infty}T_{u,v}(N)$. $\forall(x,y)\in\mathring T_{u,v}$, $0<x-\frac{u}{p}<y-\frac{v}{p}<\frac{1}{p}$. Then there exists $0<\epsilon_1, \epsilon_2, \eta_1, \eta_2<\frac{1}{p}$ such that $x=\frac{u}{p}+\epsilon_1=\frac{u+1}{p}-\eta_1$ and $y=\frac{v}{p}+\epsilon_2=\frac{v+1}{p}-\eta_2$. For each $N>0$, define $j(N)=\lfloor(N+1)x\rfloor$ and $k(N)=\lfloor(N+1)y\rfloor$. When $N$ is large enough, $j(N)=\lfloor(N+1)(\frac{u}{p}+\epsilon_1)\rfloor=uq(N)+\lfloor u\frac{l(N)}{p}+\epsilon_1(N+1)\rfloor\geq uq(N)+\min\{u,l(N)\}=a_u$, while $j(N)=\lfloor(N+1)(\frac{u+1}{p}-\eta_1)\rfloor=(u+1)q(N)+\lfloor(u+1)\frac{l(N)}{p}-\eta_1(N+1)\rfloor\leq(u+1)q(N)+\min\{u+1,l(N)\}-1=a_{u+1}-1$ for $0\leq u<p-1$, which indicates that $j(N)\in I_u(N)$. And similarly, $k(N)\in I_v(N)$ for $N$ large. Moreover, $\di\lim_{N\to\infty}\frac{\tau(j(N))}{N+1}=p\di\lim_{N\to\infty}\frac{j(N)}{N+1}-u=p\di\lim_{N\to\infty}\frac{\lfloor(N+1)x\rfloor}{N+1}-u=px-u$, similarly $\di\lim_{N\to\infty}\frac{\tau(k(N))}{N+1}=py-v$. And since $0<px-u<py-v<1$, for $N$ large enough, $0<\frac{\tau(j(N))}{N+1}<\frac{\tau(k(N))}{N+1}<1\ \Rightarrow\ 0<\tau(j(N))<\tau(k(N))\leq N$. Thus by the definition of $j(N)$ and $k(N)$, we have, for $N$ large, $(x,y)\in[\frac{j(N)}{N+1},\frac{j(N)+1}{N+1}]\times[\frac{k(N)}{N+1},\frac{k(N)+1}{N+1}]\subset\di\bigcup_{(j,k)\in L_{u,v,N}}\big[\frac{j}{N+1},\frac{j+1}{N+1}\big]\times\big[\frac{k}{N+1},\frac{k+1}{N+1}\big]=T_{u,v}(N)$, which implies that $(x,y)\in\di\liminf_{N\to\infty}T_{u,v}(N)$. \\

Inclusion, we have
\begin{align*}
\begin{split}
\mathring T_{u,v}&\subset\di\liminf_{N\to\infty}T_{u,v}(N)\subset\di\limsup_{N\to\infty}T_{u,v}(N)\subset T_{u,v}, \\
&\Rightarrow\ \di\lim_{N\to\infty}T_{u,v}(N)\setminus\partial T_{u,v}=\mathring T_{u,v}.
\end{split}
\end{align*}
\end{proof}

Let $\zeta=(\zeta_J)^t_{J\in\Ga_{m,N}}=(\tilde{s}_N(\xi_J))^t_{J\in\Ga_{m,N}}=(\tilde{s}_N(\xi_{1,j_1},\dots,\xi_{m,j_m}))^t_{J\in\Ga_{m,N}}$ be a dimension $\binom{N+m}{m}$ mean zero complex Gaussian random vector. Denote its covariance matrix by $\Si$, then $\forall J=(j_1,\dots,j_m),J'=(j'_1,\dots,j'_m)\in\Ga_{m,N}$,
\begin{align*}
\begin{split}
\Si_{J,J'}
&=\E_N(\zeta_J\bar{\zeta}_{J'})=\E_N(\tilde{s}_N(\xi_J)\overline{\tilde{s}_N(\xi_{J'})}) \\
&=\sum_{K\in\La_{m,N}}\big[\sqrt{\binom{N}{K}}\xi_J^K\big]\big[\sqrt{\binom{N}{K}}\bar{\xi}_{J'}^K\big] \\ &=\sum_{K\in\La_{m,N}}\binom{N}{K}(\xi_J\bar{\xi}_{J'})^K \\
&=(1+\xi_J\bar{\xi}_{J'})^N \\
&=(1+\xi_{1,j_1}\bar{\xi}_{1,j'_1}+\dots+\xi_{m,j_m}\bar{\xi}_{m,j'_m})^N, \\
\end{split}
\end{align*}
where $\E_N$ denotes the expectation with respect to the probability measure $\gamma_N$.
\begin{lem}\label{lem(detSi)}
With the assignment of $\xi$ as in table (\ref{table}),
\begin{align*}
\log{(\det{\Si})}=Q_{\kappa r,m}(N)+\frac{2\be_m}{p}N^{m+1}+o(N^{m+1}).
\end{align*}
\end{lem}
\begin{proof}
\begin{align*}
\Si=V_{m,N}(\xi)V_{m,N}^*(\xi),
\end{align*}
where $V_{m,N}(\xi)=\big(\sqrt{\binom{N}{K}}\xi^K_J\big)_{J\in\Ga_{m,N},\ K\in\La_{m,N}}$ is an $\binom{N+m}{m}\times\binom{N+m}{m}$ matrix.
\begin{align*}
\Rightarrow \det{\Si}=\abs{\det{V_{m,N}(\xi)}}^2=\di\prod_{K\in\La_{m,N}}\binom{N}{K}\abs{\det{W_{m,N}(\xi)}}^2
\end{align*}
By Lemma \ref{lem(detW1)},
\begin{align*}
\begin{split}
\log{(\det{\Si})}
&=\di\sum_{K\in\La_{m,N}}\log{\binom{N}{K}}+2\log{\abs{\det{W_{m,N}(\xi)}}} \\
&=\di\sum_{K\in\La_{m,N}}\log{\binom{N}{K}}+2m\binom{N+m}{m+1}\log{(\kappa r)}+\frac{2\be_m}{p}N^{m+1}+o(N^{m+1}) \\
&=\di\sum_{K\in\La_{m,N}}\log{\binom{N}{K}}+2\di\sum_{K\in\La_{m,N}}\abs{K}\log{(\kappa r)}+\frac{2\be_m}{p}N^{m+1}+o(N^{m+1}) \\
&=Q_{\kappa r,m}(N)+\frac{2\be_m}{p}N^{m+1}+o(N^{m+1}).
\end{split}
\end{align*}
\end{proof}
As $\log{\abs{\tilde{s}_N(z)}}$ is plurisubharmonic in a neighbourhood of $(\bar{D}(0,r))^m$, we have
\begin{align}\label{ineq2}
\begin{split}
&\log{\di\prod_{J\in\Ga_{m,N}}\abs{\zeta_J}} \\
=&\di\sum_{J\in\Ga_{m,N}}\log{\abs{\tilde{s}_N(\xi_J)}} \\
\leq&\di\sum_{J\in\Ga_{m,N}}\int_{\pa D(0,r)}\cdots\int_{\pa D(0,r)}\log{\abs{\tilde{s}_N(u)}}\di\prod_{i=1}^mP_r(\xi_{i,j_i},u_i)\ d\si_r(u_1)\cdots d\si_r(u_m) \\
=&(N+1)^m\int_{\pa D(0,r)}\cdots\int_{\pa D(0,r)}\log{\abs{\tilde{s}_N(u)}}
\big[\sum_{J\in\Ga_{m,N}}\prod_{i=1}^m\frac{P_r(\xi_{i,j_i},u_i)}{N+1}-\int_{H}\prod_{i=1}^mP_r(\kappa re^{2\pi\sqrt{-1}x_i},u_i)\ d_mx\big] \\
&d\si_r(u_1)\cdots d\si_r(u_m) \\
&+(N+1)^m\int_{\pa D(0,r)}\cdots\int_{\pa D(0,r)}\log{\abs{\tilde{s}_N(u)}}\int_{H}\prod_{i=1}^mP_r(\kappa re^{2\pi\sqrt{-1}x_i},u_i)\ d_mx
\ d\si_r(u_1)\cdots d\si_r(u_m) \\
=&I+II,
\end{split}
\end{align}
where $P_r(\xi,u)=\frac{r^2-\abs{\xi}^2}{\abs{u-\xi}^2}$ is the Poisson kernel of $D(0,r)$, $d\si_r$ is the Haar measure on $\pa D(0,r)$, $d_mx$ is the Lebesgue measure on $\R^m$, and
\begin{align*}
H=\di\bigcup_{0\leq t_1,\dots,t_m\leq p-1}H_{t_1,\dots,t_m}:=\di\bigcup_{0\leq t_1,\dots,t_m\leq p-1}\{x=(x_1,\dots,x_m)\in\R^m:\ 0\leq x_1-\frac{t_1}{p}\leq\dots\leq x_m-\frac{t_m}{p}\leq\frac{1}{p}\}.
\end{align*}
\begin{align}\label{ineq3}
\begin{split}
I\leq&(N+1)^m\max_{u\in(\pa D(0,r))^m}
\Abs{\sum_{J\in\Ga_{m,N}}\prod_{i=1}^m\frac{P_r(\xi_{i,j_i},u_i)}{N+1}-\int_{H}\prod_{i=1}^mP_r(\kappa re^{2\pi\sqrt{-1}x_i},u_i)\ d_mx} \\
&\times\int_{\pa D(0,r)}\cdots\int_{\pa D(0,r)}\Abs{\log{\abs{\tilde{s}_N(u)}}}\ d\si_r(u_1)\cdots d\si_r(u_m).
\end{split}
\end{align}
First let's estimate $\int_{\pa D(0,r)}\cdots\int_{\pa D(0,r)}\Abs{\log{\abs{\tilde{s}_N(u)}}}\ d\si_r(u_1)\cdots d\si_r(u_m)$.
\begin{lem}\label{lem(sup)}
$\gamma_N\big(\di\sup_{u\in(\pa D(0,r))^m}\abs{\tilde{s}_N(u)}<1\big)\leq e^{-Q_{r,m}(N)}$.
\end{lem}
\begin{proof}
\begin{align*}
\begin{split}
&\tilde{s}_N(u)=\sum_{K\in\La_{m,N}}c_K\sqrt{\binom{N}{K}}u^K \\
\Rightarrow&\frac{\pa^K}{\pa u^K}\tilde{s}_N(0)=K!\sqrt{\binom{N}{K}}c_K,
\end{split}
\end{align*}
where $\frac{\pa^K}{\pa u^K}$ refers to $\frac{\pa^{k_1}}{\pa u_1^{k_1}}\cdots\frac{\pa^{k_m}}{\pa u_1^{k_m}}$ and $K!=k_1!\cdots k_m!$.

By Cauchy's integral formula,
\begin{align*}
\frac{\pa^K}{\pa u^K}\tilde{s}_N(0)=
\frac{K!}{(2\pi\sqrt{-1})^m}\int_{\pa D(0,r)}\cdots\int_{\pa D(0,r)}\frac{\tilde{s}_N(u)}{\di\prod_{i=1}^mu_i^{k_i+1}}\ du_1\cdots du_m,
\end{align*}
\begin{align*}
\begin{split}
&\Rightarrow c_K=\binom{N}{K}^{-\half}
\frac{1}{(2\pi\sqrt{-1})^m}\int_{\pa D(0,r)}\cdots\int_{\pa D(0,r)}\frac{\tilde{s}_N(u)}{\di\prod_{i=1}^mu_i^{k_i+1}}\ du_1\cdots du_m, \\
&\Rightarrow \abs{c_K}\leq\frac{\di\sup_{u\in(\pa D(0,r))^m}\abs{\tilde{s}_N(u)}}{\sqrt{\binom{N}{K}}r^{\abs{K}}},\ \forall\ K\in\La_{m,N}.
\end{split}
\end{align*}
Therefore, $\di\sup_{u\in(\pa D(0,r))^m}\abs{\tilde{s}_N(u)}<1$ would imply that $\forall\ K\in\La_{m,N}$,
\begin{align*}
\abs{c_K}\leq\big[\binom{N}{K}r^{2\abs{K}}\big]^{-\half}.
\end{align*}
\begin{align*}
\begin{split}
\Rightarrow \gamma_N\big(\di\sup_{u\in(\pa D(0,r))^m}\abs{\tilde{s}_N(u)}<1\big)
&\leq\di\prod_{K\in\La_{m,N}}\gamma_N\Big(\abs{c_K}\leq\big[\binom{N}{K}r^{2\abs{K}}\big]^{-\half}\Big) \\
&\leq\di\prod_{K\in\La_{m,N}}\big[\binom{N}{K}r^{2\abs{K}}\big]^{-1} \\
&=e^{-Q_{r,m}(N)}.
\end{split}
\end{align*}
\end{proof}

The next lemma follows directly from the first part of Theorem 3.1 in \cite{SZZ}. But here we provide a self-contained proof without using the language of sections and metrics.
\begin{lem}\label{lem(sups_N)}
Given $U\subset\C^m$ open and bounded with $\di\sup_{z\in\bar{U}}\abs{z}=R>0$, then $\forall\ \eta>0$,
\begin{align*}
\gamma_N\{\di\sup_{z\in\bar{U}}\abs{\tilde{s}_N(z)}>(1+R^2)^{\frac{N}{2}}e^{\eta N}\}\leq e^{-e^{\eta N}},\text{ for }N\gg1.
\end{align*}
\end{lem}
\begin{proof}
By Cauchy-Schwartz inequality,
\begin{align*}
\begin{split}
\di\sup_{z\in\bar{U}}\abs{\tilde{s}_N(z)}
&=\di\sup_{z\in\bar{U}}\Big|\di\sum_{K\in\La_{m,N}}c_K\sqrt{\binom{N}{K}}z^K\Big| \\
&\leq\abs{c}\di\sup_{z\in\bar{U}}\Big[\di\sum_{K\in\La_{m,N}}\binom{N}{K}\abs{z}^{2K}\Big]^{\half} \\
&=\abs{c}\di\sup_{z\in\bar{U}}(1+\abs{z}^2)^{\frac{N}{2}} \\
&=\abs{c}(1+R^2)^{\frac{N}{2}},
\end{split}
\end{align*}
\begin{align*}
\begin{split}
\Rightarrow\ &\gamma_N\{\di\sup_{z\in\bar{U}}\abs{\tilde{s}_N(z)}>(1+R^2)^{\frac{N}{2}}e^{\eta N}\} \\
\leq&\gamma_N\{\abs{c}>e^{\eta N}\} \\
=&e^{-e^{2\eta N}}\di\sum_{k=0}^{\binom{N+m}{m}-1}\frac{e^{(2\eta N)k}}{k!},
\end{split}
\end{align*}
\begin{align*}
\begin{split}
\Rightarrow\ &\log{\gamma_N\{\di\sup_{z\in\bar{U}}\abs{\tilde{s}_N(z)}>(1+R^2)^{\frac{N}{2}}e^{\eta N}\}} \\
\leq&-e^{2\eta N}+\log{\binom{N+m}{m}}+(2\eta N)\Big[\binom{N+m}{m}-1\Big] \\
\leq&-e^{\eta N},\text{ for }N\gg1.
\end{split}
\end{align*}
\end{proof}

\begin{lem}\label{lem(int)}
$\di\int_{\pa D(0,r)}\cdots\int_{\pa D(0,r)}\Abs{\log\abs{\tilde{s}_N(u)}}\ d\si_r(u_1)\cdots d\si_r(u_m)\leq\frac{CN}{\de^m}$ for some constant $C$ outside an event of probability at most $e^{-e^N}+e^{-Q_{\kappa r,m}(N)}$.
\end{lem}
\begin{proof}
Applying Lemma \ref{lem(sups_N)} to $U=(D(0,r))^m$, we have
\begin{align}\label{ineq(sup)}
\gamma_N\{\di\sup_{u\in(\pa D(0,r))^m}\abs{\tilde{s}_N(u)}>(1+mr^2)^{\frac{N}{2}}e^{\eta N}\}
\leq\gamma_N\{\di\sup_{u\in(\bar{D}(0,r))^m}\abs{\tilde{s}_N(u)}>(1+mr^2)^{\frac{N}{2}}e^{\eta N}\}\leq e^{-e^{\eta N}}.
\end{align}
Therefore, take $\eta=1$, outside an event of probability at most $e^{-e^N}$,
\begin{align}\label{ineq4}
\begin{split}
&\log^+\abs{\tilde{s}_N(u)}\leq\half N\log(1+mr^2)+N\text{ on }(\pa D(0,r))^m, \\
\Rightarrow&\int_{\pa D(0,r)}\cdots\int_{\pa D(0,r)}\log^+\abs{\tilde{s}_N(u)}\ d\si_r(u_1)\cdots d\si_r(u_m)\leq\half N\log(1+mr^2)+N.
\end{split}
\end{align}
Applying Lemma \ref{lem(sup)} to the distinguished boundary $(\pa D(0,\kappa r))^m$, we have: outside an event of probability at most $e^{-Q_{\kappa r, m}(N)}$,
$\di\sup_{u\in(\pa D(0,\kappa r))^m}\abs{\tilde{s}_N(u)}\geq1$, i.e. $\exists\ \eta\in(\pa D(0,\kappa r))^m$ such that $\abs{\tilde{s}_N(\eta)}\geq1$,
\begin{align}{\label{ineq5}}
\begin{split}
0\leq\log{\abs{\tilde{s}_N(\eta)}}
\leq&\int_{\pa D(0,r)}\cdots\int_{\pa D(0,r)}\log{\abs{\tilde{s}_N(u)}}\di\prod_{i=1}^mP_r(\eta_i,u_i)\ d\si_r(u_1)\cdots d\si_r(u_m) \\
=&\int_{\pa D(0,r)}\cdots\int_{\pa D(0,r)}\log^+{\abs{\tilde{s}_N(u)}}\di\prod_{i=1}^mP_r(\eta_i,u_i)\ d\si_r(u_1)\cdots d\si_r(u_m) \\
&-\int_{\pa D(0,r)}\cdots\int_{\pa D(0,r)}\log^-{\abs{\tilde{s}_N(u)}}\di\prod_{i=1}^mP_r(\eta_i,u_i)\ d\si_r(u_1)\cdots d\si_r(u_m),
\end{split}
\end{align}
Since $\forall\ 1\leq i\leq m$, $\abs{\eta_i}=\kappa r=(1-\sqrt{\de})r$ and $\abs{u_i}=r$, $\frac{\sqrt{\delta}}{2}\leq P_r(\eta_i,u_i)\leq\frac{2}{\sqrt{\de}}$, (\ref{ineq5}) implies that outside an event of probability at most $e^{-Q_{\kappa r,m}(N)}$,
\begin{align}\label{ineq6}
\begin{split}
&\big(\frac{\sqrt{\delta}}{2}\big)^m\int_{\pa D(0,r)}\cdots\int_{\pa D(0,r)}\log^-{\abs{\tilde{s}_N(u)}}\ d\si_r(u_1)\cdots d\si_r(u_m) \\
\leq&\big(\frac{2}{\sqrt{\de}}\big)^m\int_{\pa D(0,r)}\cdots\int_{\pa D(0,r)}\log^+{\abs{\tilde{s}_N(u)}}\ d\si_r(u_1)\cdots d\si_r(u_m).
\end{split}
\end{align}
Combine (\ref{ineq4}) and (\ref{ineq6}), we get: outside an event of probability at most $e^{-e^N}+e^{-Q_{\kappa r,m}(N)}$,
\begin{align*}
\begin{split}
&\int_{\pa D(0,r)}\cdots\int_{\pa D(0,r)}\Abs{\log\abs{\tilde{s}_N(u)}}\ d\si_r(u_1)\cdots d\si_r(u_m) \\
=&\int_{\pa D(0,r)}\cdots\int_{\pa D(0,r)}\log^+\abs{\tilde{s}_N(u)}\ d\si_r(u_1)\cdots d\si_r(u_m)
+\int_{\pa D(0,r)}\cdots\int_{\pa D(0,r)}\log^-\abs{\tilde{s}_N(u)}\ d\si_r(u_1)\cdots d\si_r(u_m) \\
\leq&\big[1+\big(\frac{4}{\de}\big)^m\big]
\int_{\pa D(0,r)}\cdots\int_{\pa D(0,r)}\log^+\abs{\tilde{s}_N(u)}\ d\si_r(u_1)\cdots d\si_r(u_m) \\
\leq&\big[1+\big(\frac{4}{\de}\big)^m\big][\half N\log(1+mr^2)+N]=\frac{CN}{\de^m}.
\end{split}
\end{align*}
\end{proof}
The following lemma estimates $\di\max_{u\in(\pa D(0,r))^m}
\Abs{\sum_{J\in\Ga_{m,N}}\prod_{i=1}^m\frac{P_r(\xi_{i,j_i},u_i)}{N+1}-\int_{H}\prod_{i=1}^mP_r(\kappa re^{2\pi\sqrt{-1}x_i},u_i)\ d_mx}$:

\begin{lem}\label{lem(max)}
$\di\max_{u\in(\pa D(0,r))^m}
\Abs{\sum_{J\in\Ga_{m,N}}\prod_{i=1}^m\frac{P_r(\xi_{i,j_i},u_i)}{N+1}-\int_{H}\prod_{i=1}^mP_r(\kappa re^{2\pi\sqrt{-1}x_i},u_i)\ d_mx}
\leq\frac{o(1)}{\de^{\half(m+1)}}$.
\end{lem}

\begin{proof}
For all $u\in(\pa D(0,r))^m$,
\begin{align}\label{ineq7}
\begin{split}
&\Abs{\sum_{J\in\Ga_{m,N}}\prod_{i=1}^m\frac{P_r(\xi_{i,j_i},u_i)}{N+1}-\int_{H}\prod_{i=1}^mP_r(\kappa re^{2\pi\sqrt{-1}x_i},u_i)\ d_mx} \\
=&\Abs{\sum_{\tau(J)\in\Ga_{m,N}}\prod_{i=1}^m\frac{P_r(\xi_{i,\tau(j_i)},u_i)}{N+1}-\int_{H}\prod_{i=1}^mP_r(\kappa re^{2\pi\sqrt{-1}x_i},u_i)\ d_mx}
\\
\leq&\di\sum_{0\leq t_1,\dots,t_m\leq p-1}\Abs{\di\sum_{J\in I_{t_1}\times\cdots\times I_{t_m}:\ \tau(J)\in\Ga_{m,N}}
\di\prod_{i=1}^m\frac{P_r(z_{j_i},u_i)}{N+1}-\int_{H_{t_1,\dots,t_m}}\prod_{i=1}^mP_r(\kappa re^{2\pi\sqrt{-1}x_i},u_i)\ d_mx} \\
\leq&\di\sum_{0\leq t_1,\dots,t_m\leq p-1}
\Abs{\di\sum_{J\in I_{t_1}\times\cdots\times I_{t_m}:\ \tau(J)\in\Ga_{m,N}}
\di\prod_{i=1}^m\frac{P_r(z_{j_i},u_i)}{N+1}
-\int_{H_{t_1,\dots,t_m}(N)}\prod_{i=1}^mP_r(\kappa re^{2\pi\sqrt{-1}x_i},u_i)\ d_mx} \\
&+\di\sum_{0\leq t_1,\dots,t_m\leq p-1}
\Abs{\int_{H_{t_1,\dots,t_m}(N)}\prod_{i=1}^mP_r(\kappa re^{2\pi\sqrt{-1}x_i},u_i)\ d_mx
-\int_{H_{t_1,\dots,t_m}}\prod_{i=1}^mP_r(\kappa re^{2\pi\sqrt{-1}x_i},u_i)\ d_mx},
\end{split}
\end{align}
where $H_{t_1,\dots,t_m}(N)=\di\bigcup_{J\in I_{t_1}\times\cdots\times I_{t_m}:\ \tau(J)\in\Ga_{m,N}}\big[\frac{j_1}{N+1},\frac{j_1+1}{N+1}\big]\times\cdots\times\big[\frac{j_m}{N+1},\frac{j_m+1}{N+1}\big]$.

$\forall\ 0\leq t_1,\dots,t_m\leq p-1$,
\begin{align*}
\begin{split}
&\Abs{\di\sum_{J\in I_{t_1}\times\cdots\times I_{t_m}:\ \tau(J)\in\Ga_{m,N}}
\di\prod_{i=1}^m\frac{P_r(z_{j_i},u_i)}{N+1}
-\int_{H_{t_1,\dots,t_m}(N)}\prod_{i=1}^mP_r(\kappa re^{2\pi\sqrt{-1}x_i},u_i)\ d_mx} \\
\leq&\di\sum_{J\in I_{t_1}\times\cdots\times I_{t_m}:\ \tau(J)\in\Ga_{m,N}}
\int_{\big[\frac{j_1}{N+1},\frac{j_1+1}{N+1}\big]\times\cdots\times\big[\frac{j_m}{N+1},\frac{j_m+1}{N+1}\big]}
\Abs{\prod_{i=1}^mP_r(\kappa re^{2\pi\sqrt{-1}x_i},u_i)-\prod_{i=1}^mP_r(\kappa re^{2\pi\sqrt{-1}\frac{j_i}{N+1}},u_i)}\ d_mx \\
\leq&\frac{(q+1)^m}{(N+1)^m}m\di\sup_{\abs{\om}=\kappa r,\abs{u}=r}[P_r(\om,u)]^{m-1}
\di\sup_{\abs{\om}\leq\kappa r,\abs{u}=r}\Abs{\frac{\pa{P_r(\om,u)}}{\pa\om}}\frac{2\pi\kappa r}{N+1} \\
\leq&\frac{C}{p^m\de^{\half(m+1)}(N+1)},
\end{split}
\end{align*}
\begin{align}\label{ineq8}
\begin{split}
\Rightarrow&\di\sum_{0\leq t_1,\dots,t_m\leq p-1}
\Abs{\di\sum_{J\in I_{t_1}\times\cdots\times I_{t_m}:\ \tau(J)\in\Ga_{m,N}}
\di\prod_{i=1}^m\frac{P_r(z_{j_i},u_i)}{N+1}
-\int_{H_{t_1,\dots,t_m}(N)}\prod_{i=1}^mP_r(\kappa re^{2\pi\sqrt{-1}x_i},u_i)\ d_mx} \\
\leq&\frac{C}{\de^{\half(m+1)}(N+1)}=\frac{o(1)}{\de^{\half(m+1)}}
\end{split}
\end{align}
To bound the second term in (\ref{ineq7}), we need the following statement, which can be proved in a similar way as Lemma \ref{lem(detW2)}:
\begin{align*}
\di\lim_{N\to\infty}\vol_{\R^m}(H_{t_1,\dots,t_m}(N)\ \triangle\ H_{t_1,\dots,t_m})=0\text{ for any }0\leq t_1,\dots,t_m\leq p-1.
\end{align*}
Hence,
\begin{align}\label{ineq9}
\begin{split}
&\di\sum_{0\leq t_1,\dots,t_m\leq p-1}
\Abs{\int_{H_{t_1,\dots,t_m}(N)}\prod_{i=1}^mP_r(\kappa re^{2\pi\sqrt{-1}x_i},u_i)\ d_mx
-\int_{H_{t_1,\dots,t_m}}\prod_{i=1}^mP_r(\kappa re^{2\pi\sqrt{-1}x_i},u_i)\ d_mx} \\
\leq&\di\sum_{0\leq t_1,\dots,t_m\leq p-1}
\vol_{R^m}(H_{t_1,\dots,t_m}(N)\ \triangle\ H_{t_1,\dots,t_m})
\big[\di\sup_{\abs{\om}=\kappa r,\abs{u}=r}P_r(\om,u)\big]^m \\
\leq&\di\sum_{0\leq t_1,\dots,t_m\leq p-1}o(1)\big(\frac{2}{\sqrt{\de}}\big)^m \\
=&\frac{o(1)}{\de^{\half m}}.
\end{split}
\end{align}
This $o(1)$ may depend on $p$.

By (\ref{ineq7}), (\ref{ineq8}) and (\ref{ineq9}), we prove the lemma.
\end{proof}

Combine (\ref{ineq3}), Lemma \ref{lem(int)} and Lemma \ref{lem(max)}, we have: outside an event of probability at most $e^{-e^N}+e^{-Q_{\kappa r,m}(N)}$,
\begin{align*}
I\leq(N+1)^m\frac{o(1)}{\de^{\half(m+1)}}\frac{CN}{\de^m}=\frac{o(N^{m+1})}{\de^{\frac{3}{2}m+\half}}.
\end{align*}

By changing the order of integration,
\begin{align*}
II=(N+1)^m\int_{H}
\int_{\pa D(0,r)}\cdots\int_{\pa D(0,r)}\log{\abs{\tilde{s}_N(u)}}\prod_{i=1}^mP_r(\kappa re^{2\pi\sqrt{-1}x_i},u_i)\ d\si_r(u_1)\cdots d\si_r(u_m)\ d_mx.
\end{align*}
If $\tilde{s}_N$ is nonvanishing on $(\bar{D}(0,r))^m$, $\log\abs{\tilde{s}_N(u)}$ is harmonic in $u_i\in$ a neighbourhood of $\bar{D}(0,r)$ for each fixed $(u_1,\dots,\hat{u}_i,\dots,u_m)\in(\bar{D}(0,r))^{m-1}$. Applying mean value theorem for harmonic functions, we get
\begin{align*}
\begin{split}
II
=&(N+1)^m\times \\
&\int_{H}\int_{\pa D(0,r)}\cdots\int_{\pa D(0,r)}\log{\abs{\tilde{s}_N(\kappa re^{2\pi\sqrt{-1}x_1},u_2,\dots,u_m)}}
\prod_{i=2}^mP_r(\kappa re^{2\pi\sqrt{-1}x_i},u_i)\ d\si_r(u_2)\cdots d\si_r(u_m)\ d_mx \\
=&\dots \\
=&(N+1)^m\int_{H}\log{\abs{\tilde{s}_N(\kappa re^{2\pi\sqrt{-1}x_1},\dots,\kappa re^{2\pi\sqrt{-1}x_m})}}\ d_mx.
\end{split}
\end{align*}
Denote
\begin{align}\label{Xi}
\Xi=\di\int_{H}\log{\abs{\tilde{s}_N(\kappa re^{2\pi\sqrt{-1}x_1},\dots,\kappa re^{2\pi\sqrt{-1}x_m})}}\ d_mx,
\end{align}
which is a complex random variable. Thus we have proved:
\begin{lem}\label{lem(logProdzeta)}
If $\tilde{s}_N$ is nonvanishing on $(\bar{D}(0,r))^m$, then outside an event of probability at most $e^{-e^N}+e^{-Q_{\kappa r,m}(N)}$,
\begin{align*}
\log{\di\prod_{J\in\Ga_{m,N}}\abs{\zeta_J}}\leq\frac{o(N^{m+1})}{\de^{\frac{3}{2}m+\half}}+(N+1)^m\Xi.
\end{align*}
\end{lem}

Replace $\Ga_{m,N}=\{J=(j_1,\dots,j_m)\in[0,N]^m\cap\Z^m:\ 0\leq j_1\leq\dots\leq j_m\leq N\}$ by $\Ga^{(\varrho)}_{m,N}=\{J=(j_1,\dots,j_m)\in[0,N]^m\cap\Z^m:\ 0\leq j_{\varrho(1)}\leq\dots\leq j_{\varrho(m)}\leq N\}$, where $\varrho$ can be any element in $S_m$, the permutation group of $m$ letters. Then similar results hold and we have counterparts for Lemma \ref{lem(detSi)} and Lemma \ref{lem(logProdzeta)}, which we state without proof.

\begin{lem}\label{lem(detSi')}
Denote the covariance matrix of the random vector $(\zeta^{(\varrho)}_J=\tilde{s}_N(\xi_J))^t_{J\in\Ga^{(\varrho)}_{m,N}}$ by $\Si^{(\varrho)}$. Then $\log{(\det{\Si^{(\varrho)}})}=Q_{\kappa r,m}(N)+\frac{2\be_m}{p}N^{m+1}+o(N^{m+1})$.
\end{lem}

$\forall \varrho\in S_m$, denote
\begin{align*}
\begin{split}
H^{(\varrho)}
&=\di\bigcup_{0\leq t_1,\dots,t_m\leq p-1}H^{(\varrho)}_{t_1,\dots,t_m} \\
&:=\di\bigcup_{0\leq t_1,\dots,t_m\leq p-1}\{x=(x_1,\dots,x_m)\in\R^m:\ 0\leq x_{\varrho(1)}-\frac{t_{\varrho(1)}}{p}\leq\dots\leq
x_{\varrho(m)}-\frac{t_{\varrho(m)}}{p}\leq\frac{1}{p}\}
\end{split}
\end{align*}
and the random variable
\begin{align*}
\Xi^{(\varrho)}=\di\int_{H^{(\varrho)}}\log{\abs{\tilde{s}_N(\kappa re^{2\pi\sqrt{-1}x_1},\dots,\kappa re^{2\pi\sqrt{-1}x_m})}}\ d_mx.
\end{align*}
Then
\begin{lem}\label{lem(logProdzeta')}
If $\tilde{s}_N$ is nonvanishing on $(\bar{D}(0,r))^m$, then outside an event of probability at most $e^{-e^N}+e^{-Q_{\kappa r,m}(N)}$,
\begin{align*}
\log{\di\prod_{J\in\Ga^{(\varrho)}_{m,N}}\abs{\zeta^{(\varrho)}_J}}\leq\frac{o(N^{m+1})}{\de^{\frac{3}{2}m+\half}}+(N+1)^m\Xi^{(\varrho)}.
\end{align*}
\end{lem}

If $\tilde{s}_N$ is nonvanishing on $(\bar{D}(0,r))^m$,
\begin{align*}
\begin{split}
\di\sum_{\varrho\in S_m}\Xi^{(\varrho)}
=&\di\sum_{\varrho\in S_m}\di\int_{H^{(\varrho)}}\log\abs{\tilde{s}_N(\kappa re^{2\pi\sqrt{-1}x_1},\dots,\kappa re^{2\pi\sqrt{-1}x_m})}\ d_mx \\
=&\di\int_{\di\bigcup_{\varrho\in S_m}H^{(\varrho)}}\log\abs{\tilde{s}_N(\kappa re^{2\pi\sqrt{-1}x_1},\dots,\kappa re^{2\pi\sqrt{-1}x_m})}\ d_mx \\
=&\di\int_0^1\cdots\int_0^1\log\abs{\tilde{s}_N(\kappa re^{2\pi\sqrt{-1}x_1},\dots,\kappa re^{2\pi\sqrt{-1}x_m})}\ dx_1\cdots dx_m \\
=&\di\int_{\pa D(0,\kappa r)}\cdots\int_{\pa D(0,\kappa r)}\log\abs{\tilde{s}_N(\om_1,\dots,\om_m)}\ d\si_{\kappa r}(\om_1)\cdots d\si_{\kappa r}(\om_m) \\
=&\log\abs{\tilde{s}_N(0,\dots,0)} \\
=&\log\abs{c_{(0,\dots,0)}},
\end{split}
\end{align*}
the second equality holds because for distinct $\varrho_1, \varrho_2\in S_m$, $H^{(\varrho_1)}\cap H^{(\varrho_2)}$ is of $m$-dimensional Lebesgue measure zero.
\begin{proof}[Proof of the upper bound in Theorem \ref{main2}]
\begin{align*}
\begin{split}
P_{0,m}(r,N)
=&\gamma_N\{0\not\in\tilde{s}_N\big((\bar{D}(0,r))^m\big)\} \\
=&\gamma_N\big\{(\log\abs{c_{(0,\dots,0)}}>2m!\log{N})\cap\Big(0\not\in\tilde{s}_N\big((\bar{D}(0,r))^m\big)\Big)\big\} \\
&+\gamma_N\big\{(\log\abs{c_{(0,\dots,0)}}\leq 2m!\log{N})\cap\Big(0\not\in\tilde{s}_N\big((\bar{D}(0,r))^m\big)\Big)\big\} \\
\leq&\gamma_N(\abs{c_{(0,\dots,0)}}>N^{2m!})+\gamma_N\big\{(\di\sum_{\varrho\in S_m}\Xi^{(\varrho)}\leq 2m!\log{N})\cap\Big(0\not\in\tilde{s}_N\big((\bar{D}(0,r))^m\big)\Big)\big\} \\
\leq&e^{-N^{4m!}}+\gamma_N\big\{\di\bigcup_{\varrho\in S_m}(\Xi^{(\varrho)}\leq 2\log{N})\cap\Big(0\not\in\tilde{s}_N\big((\bar{D}(0,r))^m\big)\Big)\big\} \\
\leq&e^{-N^{4m!}}+\di\sum_{\varrho\in S_m}\gamma_N\big\{(\Xi^{(\varrho)}\leq 2\log{N})\cap\Big(0\not\in\tilde{s}_N\big((\bar{D}(0,r))^m\big)\Big)\big\}.
\end{split}
\end{align*}
Lemma \ref{lem(logProdzeta)} implies
\begin{align*}
\begin{split}
&\gamma_N\{(\Xi\leq 2\log{N})\cap\Big(0\not\in\tilde{s}_N\big((\bar{D}(0,r))^m\big)\Big)\} \\
\leq&e^{-e^N}+e^{-Q_{\kappa r,m}(N)}+\gamma_N\{\log{\di\prod_{J\in\Ga_{m,N}}\abs{\zeta_J}}\leq\frac{o(N^{m+1})}{\de^{\frac{3}{2}m+\half}}+2(N+1)^m\log{N}\} \\
=&e^{-e^N}+e^{-Q_{\kappa r,m}(N)}
+\gamma_N\big\{\di\prod_{J\in\Ga_{m,N}}\abs{\zeta_J}\leq\exp\{\frac{o(N^{m+1})}{\de^{\frac{3}{2}m+\half}}+2(N+1)^m\log{N}\}\big\}.
\end{split}
\end{align*}
Denote
\begin{align*}
\ecal_{m,N}=\bigg\{\zeta=(\zeta_J)_{J\in\Ga_{m,N}}\in\C^{\binom{N+m}{m}}:
\di\prod_{J\in\Ga_{m,N}}\abs{\zeta_J}\leq\exp\Big\{\frac{o(N^{m+1})}{\de^{\frac{3}{2}m+\half}}+2(N+1)^m\log{N}\Big\}\bigg\},
\end{align*}
and
\begin{align*}
\fcal_{m,N}=\big\{\zeta=(\zeta_J)_{J\in\Ga_{m,N}}\in\ecal_{m,N}:\abs{\zeta_J}\leq(2+2mr^2)^{\frac{N}{2}},\ \forall\ J\in\Ga_{m,N}\big\}\subset\ecal_{m,N},
\end{align*}
both of which can be treated as subsets in $\C^{\binom{N+m}{m}}$ and events in the probability space $\big(H^0(\CP^m,\ocal(N)),\gamma_N\big)$.
Thus,
\begin{align}\label{ineq11}
\begin{split}
\gamma_N\big\{(\Xi\leq 2\log{N})\cap\Big(0\not\in\tilde{s}_N\big((\bar{D}(0,r))^m\big)\Big)\big\}
\leq&e^{-e^N}+e^{-Q_{\kappa r,m}(N)}+\gamma_N(\ecal_{m,N}) \\
\leq&e^{-e^N}+e^{-Q_{\kappa r,m}(N)}+\gamma_N(\ecal_{m,N}\setminus\fcal_{m,N})+\gamma_N(\fcal_{m,N}).
\end{split}
\end{align}
\begin{align}\label{ineq(E-F)}
\begin{split}
\gamma_N(\ecal_{m,N}\setminus\fcal_{m,N})
\leq&\gamma_N\big\{\abs{\zeta_J}>(2+2mr^2)^{\frac{N}{2}}\text{ for some }J\in\Ga_{m,N}\big\} \\
\leq&\gamma_N\big\{\di\sup_{\om\in(\pa D(0,\kappa r))^m}\abs{\tilde{s}_N(\om)}>(2+2mr^2)^{\frac{N}{2}}\big\} \\
\leq&\gamma_N\big\{\di\sup_{\om\in(\bar{D}(0,r))^m}\abs{\tilde{s}_N(\om)}>(1+mr^2)^{\frac{N}{2}}2^{\frac{N}{2}}\big\} \\
\leq&e^{-2^{\frac{N}{2}}},
\end{split}
\end{align}
where the last inequality is due to Lemma \ref{lem(sups_N)}.
\begin{align*}
\begin{split}
\gamma_N(\fcal_{m,N})
=&\frac{1}{\pi^{\binom{N+m}{m}}\det{\Si}}\di\int_{\fcal_{m,N}}e^{-\zeta^*\Si^{-1}\zeta}\ d_{2\binom{N+m}{m}}\zeta \\
\leq&\exp\Big\{-\big[Q_{\kappa r,m}(N)+\frac{2\be_m}{p}N^{m+1}\big]+o(N^{m+1})\Big\}\pi^{-\binom{N+m}{m}}\vol_{\C^{\binom{N+m}{m}}}(\fcal_{m,N})
\end{split}
\end{align*}
by Lemma \ref{lem(detSi)}. Change into polar coordinates and denote
\begin{align*}
\begin{split}
&\vol_{\R^{\binom{N+m}{m}}}(\fcal_{m,N}) \\
=&\vol_{\R^{\binom{N+m}{m}}}\Big\{(x_J)_{J\in\Ga_{m,N}}\in\big[0,(2+2mr^2)^{\frac{N}{2}}\big]^{\binom{N+m}{m}}:
\di\prod_{J\in\Ga_{m,N}}x_J\leq\exp\big\{\frac{o(N^{m+1})}{\de^{\frac{3}{2}m+\half}}+2(N+1)^m\log{N}\big\}\Big\},
\end{split}
\end{align*}
\begin{align*}
\begin{split}
\Rightarrow &\gamma_N(\fcal_{m,N}) \\
\leq&2^{\binom{N+m}{m}}\exp\Big\{-\big[Q_{\kappa r,m}(N)+\frac{2\be_m}{p}N^{m+1}\big]+o(N^{m+1})\Big\}
\exp\big\{\frac{o(N^{m+1})}{\de^{\frac{3}{2}m+\half}}+2(N+1)^m\log{N}\big\}\vol_{\R^{\binom{N+m}{m}}}(\fcal_{m,N}) \\
=&2^{\binom{N+m}{m}}\exp\Big\{-\big[Q_{\kappa r,m}(N)+\frac{2\be_m}{p}N^{m+1}\big]+\frac{o(N^{m+1})}{\de^{\frac{3}{2}m+\half}}\Big\}\vol_{\R^{\binom{N+m}{m}}}(\fcal_{m,N}).
\end{split}
\end{align*}
Since $\binom{N+m}{m}\frac{N}{2}\log(2+2mr^2)-\big[\frac{o(N^{m+1})}{\de^{\frac{3}{2}m+\half}}+2(N+1)^m\log{N}\big]>\binom{N+m}{m}$ for $N$ large(up to now $p,\de$ are constants), we can apply Lemma 4.6 in \cite{N1} and get:
\begin{align*}
\begin{split}
&\vol_{\R^{\binom{N+m}{m}}}(\fcal_{m,N}) \\
\leq&\frac{\exp\Big\{\frac{o(N^{m+1})}{\de^{\frac{3}{2}m+\half}}+2(N+1)^m\log{N}\Big\}}{[\binom{N+m}{m}-1]!}
\Big\{\binom{N+m}{m}\frac{N}{2}\log(2+2mr^2)-\big[\frac{o(N^{m+1})}{\de^{\frac{3}{2}m+\half}}+2(N+1)^m\log{N}\big]\Big\}^{\binom{N+m}{m}} \\
\leq&\frac{\exp\Big\{\frac{o(N^{m+1})}{\de^{\frac{3}{2}m+\half}}+2(N+1)^m\log{N}\Big\}}{2^{\binom{N+m}{m}}[\binom{N+m}{m}-1]!}
\big[N\binom{N+m}{m}\log(2+2mr^2)\big]^{\binom{N+m}{m}}
\end{split}
\end{align*}
\begin{align*}
\begin{split}
\Rightarrow \gamma_N(\fcal_{m,N})
\leq&\frac{\exp\Big\{\frac{o(N^{m+1})}{\de^{\frac{3}{2}m+\half}}+2(N+1)^m\log{N}-[Q_{\kappa r,m}(N)+\frac{2\be_m}{p}N^{m+1}]\Big\}}
{[\binom{N+m}{m}-1]!} \\
&\times\big[N\binom{N+m}{m}\log(2+2mr^2)\big]^{\binom{N+m}{m}},
\end{split}
\end{align*}
\begin{align}\label{ineq(F)}
\begin{split}
\Rightarrow\log{\gamma_N(\fcal_{m,N})}
\leq&\frac{o(N^{m+1})}{\de^{\frac{3}{2}m+\half}}+2(N+1)^m\log{N}-[Q_{\kappa r,m}(N)+\frac{2\be_m}{p}N^{m+1}] \\
&+\binom{N+m}{m}\log{\big[N\binom{N+m}{m}\log(2+2mr^2)\big]}-\log{\big[\binom{N+m}{m}-1\big]!} \\
&=-Q_{\kappa r,m}(N)-\frac{2\be_m}{p}N^{m+1}+\frac{o(N^{m+1})}{\de^{\frac{3}{2}m+\half}}.
\end{split}
\end{align}
By Lemma \ref{lem(Q)}, (\ref{ineq11}), (\ref{ineq(E-F)}) and (\ref{ineq(F)}),
\begin{align*}
\begin{split}
&\gamma_N\big\{(\Xi\leq 2\log{N})\cap\Big(0\not\in\tilde{s}_N\big((\bar{D}(0,r))^m\big)\Big)\big\} \\
\leq&e^{-e^N}+e^{-Q_{\kappa r,m}(N)}+e^{-2^{\frac{N}{2}}}+\exp\big\{-Q_{\kappa r,m}(N)-\frac{2\be_m}{p}N^{m+1}+\frac{o(N^{m+1})}{\de^{\frac{3}{2}m+\half}}\big\} \\
\leq&\exp\Big\{-\min\big\{\frac{2m\log(\kappa r)}{(m+1)!}+\frac{1}{m!}\di\sum_{k=2}^{m+1}\frac{1}{k},\frac{2m\log(\kappa r)}{(m+1)!}+\frac{1}{m!}\di\sum_{k=2}^{m+1}\frac{1}{k}+\frac{2\be_m}{p}\big\}N^{m+1}+\frac{o(N^{m+1})}{\de^{\frac{3}{2}m+\half}}\Big\}.
\end{split}
\end{align*}
Similarly, $\forall\ \varrho\in S_m$,
\begin{align*}
&\gamma_N\big\{(\Xi^{(\varrho)}\leq 2\log{N})\cap\Big(0\not\in\tilde{s}_N\big((\bar{D}(0,r))^m\big)\Big)\big\} \\
\leq&\exp\Big\{-\min\big\{\frac{2m\log(\kappa r)}{(m+1)!}+\frac{1}{m!}\di\sum_{k=2}^{m+1}\frac{1}{k},\frac{2m\log(\kappa r)}{(m+1)!}+\frac{1}{m!}\di\sum_{k=2}^{m+1}\frac{1}{k}+\frac{2\be_m}{p}\big\}N^{m+1}+\frac{o(N^{m+1})}{\de^{\frac{3}{2}m+\half}}\Big\},
\end{align*}
\begin{align*}
\Rightarrow&P_{0,m}(r,N) \\
\leq&e^{-N^{4m!}}
+m!\exp\Big\{-\min\big\{\frac{2m\log(\kappa r)}{(m+1)!}+\frac{1}{m!}\di\sum_{k=2}^{m+1}\frac{1}{k},\frac{2m\log(\kappa r)}{(m+1)!}+\frac{1}{m!}\di\sum_{k=2}^{m+1}\frac{1}{k}+\frac{2\be_m}{p}\big\}N^{m+1}+\frac{o(N^{m+1})}{\de^{\frac{3}{2}m+\half}}\Big\} \\
=&\exp\Big\{-\min\big\{\frac{2m\log(\kappa r)}{(m+1)!}+\frac{1}{m!}\di\sum_{k=2}^{m+1}\frac{1}{k},\frac{2m\log(\kappa r)}{(m+1)!}+\frac{1}{m!}\di\sum_{k=2}^{m+1}\frac{1}{k}+\frac{2\be_m}{p}\big\}N^{m+1}+\frac{o(N^{m+1})}{\de^{\frac{3}{2}m+\half}}\Big\}, \\
\Rightarrow&\log{P_{0,m}(r,N)}\leq
-\min\big\{\frac{2m\log(\kappa r)}{(m+1)!}+\frac{1}{m!}\di\sum_{k=2}^{m+1}\frac{1}{k},\frac{2m\log(\kappa r)}{(m+1)!}+\frac{1}{m!}\di\sum_{k=2}^{m+1}\frac{1}{k}+\frac{2\be_m}{p}\big\}N^{m+1}+\frac{o(N^{m+1})}{\de^{\frac{3}{2}m+\half}}, \\
\Rightarrow&\limsup_{N\to\infty}\frac{\log{P_{0,m}(r,N)}}{N^{m+1}}\leq-\min\big\{\frac{2m\log(\kappa r)}{(m+1)!}+\frac{1}{m!}\di\sum_{k=2}^{m+1}\frac{1}{k},\frac{2m\log(\kappa r)}{(m+1)!}+\frac{1}{m!}\di\sum_{k=2}^{m+1}\frac{1}{k}+\frac{2\be_m}{p}\big\}.
\end{align*}
Let $p\to\infty$, then
\begin{align*}
\Rightarrow\limsup_{N\to\infty}\frac{\log{P_{0,m}(r,N)}}{N^{m+1}}\leq-\big[\frac{2m\log(\kappa r)}{(m+1)!}+\frac{1}{m!}\di\sum_{k=2}^{m+1}\frac{1}{k}\big].
\end{align*}
Let $\de\to0+$, then $\kappa=1-\sqrt{\de}\to1$,
\begin{align*}
\Rightarrow\limsup_{N\to\infty}\frac{\log{P_{0,m}(r,N)}}{N^{m+1}}\leq-\big[\frac{2m\log{r}}{(m+1)!}+\frac{1}{m!}\di\sum_{k=2}^{m+1}\frac{1}{k}\big].
\end{align*}
\begin{align*}
\Rightarrow\log{P_{0,m}(r,N)}\leq-\big[\frac{2m\log{r}}{(m+1)!}+\frac{1}{m!}\di\sum_{k=2}^{m+1}\frac{1}{k}\big]N^{m+1}+o(N^{m+1}).
\end{align*}
\end{proof}

\section{Proof of Theorem \ref{main1}}
\subsection{Lower bound}
\begin{definition}
\begin{align*}
\begin{split}
\La_{m,N}(r)&:=\Big\{K\in\La_{m,N}:\ \binom{N}{K}r^{2\abs{K}}\geq1\Big\}\subset\La_{m,N}, \\
R_{r,m}(N)&:=\di\sum_{K\in\La_{m,N}(r)}\log{\big[\binom{N}{K}r^{2\abs{K}}\big]}.
\end{split}
\end{align*}
\end{definition}
\begin{lem}\label{lem(lowerbound)}
$\log{P_{0,m}(r,N)}\geq-R_{r,m}(N)+o(N^{m+1})$.
\end{lem}
\begin{proof}
Consider the following event $\Om_{r,m,N}$:
\begin{align*}
\begin{split}
&(i)\ \abs{c_{(0,\dots,0)}}\geq\sqrt{N}, \\
&(ii)\ \abs{c_K}\leq\frac{1}{2\sqrt{N}\sqrt{\binom{N}{K}}r^{\abs{K}}\binom{\abs{K}+m-1}{m-1}},\ K\in\La_{m,N}(r)\backslash\{(0,\dots,0)\}, \\
&(iii)\ \abs{c_K}\leq\frac{1}{2\sqrt{N}\binom{\abs{K}+m-1}{m-1}},\ K\in\La_{m,N}\backslash\La_{m,N}(r).
\end{split}
\end{align*}
Then when $\Om_{r,m,N}$ occurs, $\forall\ z\in(\bar{D}(0,r))^m$,
\begin{align*}
\begin{split}
\abs{\tilde{s}_N(z)}
&\geq\sqrt{N}-\sum_{K\in\La_{m,N}(r)\backslash\{(0,\dots,0)\}}\frac{\sqrt{\binom{N}{K}}r^{\abs{K}}}{2\sqrt{N}\sqrt{\binom{N}{K}}r^{\abs{K}}\binom{\abs{K}+m-1}{m-1}}
-\sum_{K\in\La_{m,N}\backslash\La_{m,N}(r)}\frac{1}{2\sqrt{N}\binom{\abs{K}+m-1}{m-1}} \\
&=\sqrt{N}-\sum_{K\in\La_{m,N}\backslash\{(0,\dots,0)\}}\frac{1}{2\sqrt{N}\binom{\abs{K}+m-1}{m-1}} \\
&=\sqrt{N}-\sum_{k=1}^{N}\frac{1}{2\sqrt{N}} \\
&=\half\sqrt{N}>0.
\end{split}
\end{align*}
Thus,
\begin{align*}
\begin{split}
P_{0,m}(r,N)\geq&\gamma_N(\Om_{r,m,N}) \\
=&\gamma_N(\abs{c_{(0,\dots,0)}}\geq\sqrt{N})
\prod_{K\in\La_{m,N}(r)\backslash\{(0,\dots,0)\}}\gamma_N\bigg(\abs{c_K}\leq\frac{1}{2\sqrt{N}\sqrt{\binom{N}{K}}r^{\abs{K}}\binom{\abs{K}+m-1}{m-1}}\bigg) \\
&\times\prod_{K\in\La_{m,N}\backslash\La_{m,N}(r)}\gamma_N\bigg(\abs{c_K}\leq\frac{1}{2\sqrt{N}\binom{\abs{K}+m-1}{m-1}}\bigg) \\
\geq&e^{-N}
\prod_{K\in\La_{m,N}(r)\backslash\{(0,\dots,0)\}}\frac{1}{8N\binom{N}{K}r^{2\abs{K}}{\binom{\abs{K}+m-1}{m-1}}^2}
\prod_{K\in\La_{m,N}\backslash\La_{m,N}(r)}\frac{1}{8N{\binom{\abs{K}+m-1}{m-1}}^2},
\end{split}
\end{align*}
\begin{align*}
\begin{split}
\Rightarrow\log{P_{0,m}(r,N)}
\geq&-N-\di\sum_{K\in\La_{m,N}(r)\backslash\{(0,\dots,0)\}}\log{\big[\binom{N}{K}r^{2\abs{K}}\big]}
-\di\sum_{K\in\La_{m,N}(r)\backslash\{(0,\dots,0)\}}\log{\big[8N{\binom{\abs{K}+m-1}{m-1}}^2\big]} \\
=&-\di\sum_{K\in\La_{m,N}(r)\backslash\{(0,\dots,0)\}}\log{\big[\binom{N}{K}r^{2\abs{K}}\big]}+o(N^{m+1}) \\
=&-R_{r,m}(N)+o(N^{m+1}).
\end{split}
\end{align*}
\end{proof}

\subsection{Upper bound}

For some $\al\in(0,1]$, we can define the index sets $\La_{m,\lfloor\al N\rfloor}$, $\Ga_{m,\lfloor\al N\rfloor}$ and the $\binom{\lfloor\al N\rfloor+m}{m}\times\binom{\lfloor\al N\rfloor+m}{m}$ matrix
\begin{align*}
W_{m,\lfloor\al N\rfloor}(\xi)=(\xi_J^K)_{J\in\Ga_{m,\lfloor\al N\rfloor},\ K\in\La_{m,\lfloor\al N\rfloor}}.
\end{align*}
We also assign the values of the variables $(\xi_{i,j})_{0\leq i\leq m,\ 0\leq j\leq\lfloor\al N\rfloor}$ to be the points on $\pa D(0,\kappa r)$ in a way similar to Section 4 except that we replace $N$ by $\lfloor\al N\rfloor$. Then we have the following lemma.
\begin{lem}
\begin{align*}
\log\abs{\det{W_{m,\lfloor\al N\rfloor}(\xi)}}=m\binom{\lfloor\al N\rfloor+m}{m+1}\log{(\kappa r)}+\frac{\be_m}{p}(\lfloor\al N\rfloor)^{m+1}+o(N^{m+1}).
\end{align*}
\end{lem}
$\zeta=(\zeta_J)^t_{J\in\Ga_{m,\lfloor\al N\rfloor}}=(\tilde{s}_N(\xi_J))^t_{J\in\Ga_{m,\lfloor\al N\rfloor}}$ is a dimension $\binom{\lfloor\al N\rfloor+m}{m}$ mean zero complex Gaussian random vector with covariance matrix
\begin{align*}
\Si=V_{m,N,\al}(\xi)V^*_{m,N,\al}(\xi),
\end{align*}
where $V_{m,N,\al}(\xi)=\big(\sqrt{\binom{N}{K}}\xi^K_J\big)_{J\in\Ga_{m,\lfloor\al N\rfloor},\ K\in\La_{m,N}}$ is an $\binom{\lfloor\al N\rfloor+m}{m}\times\binom{N+m}{m}$ matrix.
\begin{definition}
$Q_{r,m,\al}(N):=\di\sum_{K\in\La_{m,\lfloor\al N\rfloor}}\log{\big[\binom{N}{K}r^{2\abs{K}}\big]}$.
\end{definition}
\begin{lem}\label{lem(detSi'')}
$\log{\det{\Si}}\geq Q_{\kappa r,m,\al}(N)+\frac{2\be_m}{p}(\lfloor\al N\rfloor)^{m+1}+o(N^{m+1})$.
\end{lem}
\begin{proof}
By Cauchy-Binet identity,
\begin{align*}
\begin{split}
\det{\Si}=&\di\sum_{M:\ \binom{\lfloor\al N\rfloor+m}{m}\times\binom{\lfloor\al N\rfloor+m}{m}\text{ minor of }V_{m,N,\al}(\xi)}\abs{\det{M}}^2 \\
      \geq&\Abs{\det{\big(\sqrt{\binom{N}{K}}\xi^K_J\big)_{J\in\Ga_{m,\lfloor\al N\rfloor},\ K\in\La_{m,\lfloor\al N\rfloor}}}}^2 \\
         =&\di\prod_{K\in\La_{m,\lfloor\al N\rfloor}}\binom{N}{K}\abs{\det{W_{m,\lfloor\al N\rfloor}(\xi)}}^2
\end{split}
\end{align*}
\begin{align*}
\begin{split}
\Rightarrow\log{\det{\Si}}
&\geq\di\sum_{K\in\La_{m,\lfloor\al N\rfloor}}\log{\binom{N}{K}}+2m\binom{\lfloor\al N\rfloor+m}{m+1}\log{(\kappa r)}+\frac{2\be_m}{p}(\lfloor\al N\rfloor)^{m+1}+o(N^{m+1}) \\
&=\di\sum_{K\in\La_{m,\lfloor\al N\rfloor}}\log{\big[\binom{N}{K}(\kappa r)^{2\abs{K}}\big]}+\frac{2\be_m}{p}(\lfloor\al N\rfloor)^{m+1}+o(N^{m+1}) \\
&=Q_{\kappa r,m,\al}(N)+\frac{2\be_m}{p}(\lfloor\al N\rfloor)^{m+1}+o(N^{m+1}).
\end{split}
\end{align*}
\end{proof}
The following lemma is a counterpart of Lemma \ref{lem(logProdzeta)}. The proof is similar.
\begin{lem}\label{lem(logProdzeta'')}
If $\tilde{s}_N$ is nonvanishing on $(\bar{D}(0,r))^m$, then outside an event of probability at most $e^{-e^N}+e^{-R_{\kappa r,m}(N)}$,
\begin{align*}
\log{\di\prod_{J\in\Ga_{m,\lfloor \al N\rfloor}}\abs{\zeta_J}}\leq\frac{o(N^{m+1})}{\de^{\frac{3}{2}m+\half}}+(\lfloor\al N\rfloor+1)^m\Xi,
\end{align*}
where the complex random variable $\Xi$ is defined in (\ref{Xi}).
\end{lem}
By playing the same trick of permutation as in Section 4, we can get an upper bound estimate for $P_{0,m}(r,N)$:
\begin{align}\label{upperbound}
P_{0,m}(r,N)\leq e^{-N^{4m!}}
+m!\bigg\{e^{-e^N}+e^{-R_{\kappa r,m}(N)}+e^{-2^{\frac{N}{2}}}
+\exp{\Big[-Q_{\kappa r,m,\al}(N)-\frac{2\be_m}{p}(\lfloor\al N\rfloor)^{m+1}+\frac{o(N^{m+1})}{\de^{\frac{3}{2}m+\half}}\Big]}\bigg\}.
\end{align}

\subsection{Punch line of the proof}
In order to prove Theorem \ref{main1}, it suffices to compute $R_{r,m}(N)$ and $Q_{r,m,\al}(N)$ asymptotically. We follow the same idea in Lemma \ref{lem(Q)}. \\

The scaled lattice $\frac{1}{N}\La_{m,N}(r)$ corresponds to the set
\begin{align*}
\{x=(x_1,\dots,x_m)\in\Si_m:\ E_r(x)\geq0\}
\end{align*}
and $\frac{1}{N}\La_{r,m,\al}(N)$ corresponds to the set
\begin{align*}
\{x=(x_1,\dots,x_m)\in\R^{m+}:\ \di\sum_{i=1}^{m}x_i\leq\al\leq1\}.
\end{align*}
So we have
\begin{align}\label{R_{r,m}}
R_{r,m}(N)=\di\sum_{K\in\La_{m,N}(r)}\log{\big[\binom{N}{K}r^{2\abs{K}}\big]}=N^{m+1}\int_{x\in\Si_m:\ E_r(x)\geq0}E_r(x)\ d_mx+o(N^{m+1}),
\end{align}
\begin{align}\label{Q_{r,m,al}}
Q_{r,m,\al}(N)&=\di\sum_{K\in\La_{m,\lfloor\al N\rfloor}}\log{\big[\binom{N}{K}r^{2\abs{K}}\big]}=N^{m+1}\int_{x\in\R^{m+}:\ \sum_{i=1}^{m}x_i\leq\al}E_r(x)\ d_mx+o(N^{m+1}).
\end{align}
Moreover, if we go through the proof of Lemma \ref{lem(Q)}, we find that the $o(N^{m+1})$ terms in (\ref{R_{r,m}}) and (\ref{Q_{r,m,al}}) are uniform if $r\leq c$ for some constant $c>0$, which implies that when $r$ is replaced by $\kappa r=(1-\sqrt{\de})r$, the remainder won't depend on $\de$.

\begin{proof}[Proof of Theorem \ref{main1}]
The lower bound proof is already implied by Lemma \ref{lem(lowerbound)} and (\ref{R_{r,m}}).
To prove the upper bound, by (\ref{upperbound}) and (\ref{Q_{r,m,al}}),
\begin{align*}
\begin{split}
&\log{P_{0,m}(r,N)} \\
\leq&-N^{m+1}\min{\Big\{\int_{x\in\Si_m:\ E_{\kappa r}(x)\geq0}E_{\kappa r}(x)\ d_mx,\ \int_{x\in\R^{m+}:\ \sum_{i=1}^{m}x_i\leq\al}E_{\kappa r}(x)\ d_mx+\frac{2\be_m\al^{m+1}}{p}\Big\}}+\frac{o(N^{m+1})}{\de^{\frac{3}{2}m+\half}}.
\end{split}
\end{align*}
Similar as in Section 4, we can get
\begin{align*}
\begin{split}
\log{P_{0,m}(r,N)}
&\leq-N^{m+1}\min{\Big\{\int_{x\in\Si_m:\ E_r(x)\geq0}E_r(x)\ d_mx, \int_{x\in\R^{m+}:\ \sum_{i=1}^{m}x_i\leq\al}E_r(x)\ d_mx\Big\}}+o(N^{m+1}) \\
&=-N^{m+1}\int_{x\in\R^{m+}:\ \sum_{i=1}^{m}x_i\leq\al}E_r(x)\ d_mx+o(N^{m+1}).
\end{split}
\end{align*}
 It amounts to find a proper $\al_0=\al_0(r,m)\in(0,1]$ which maximize $\int_{x\in\R^{m+}:\ \sum_{i=1}^{m}x_i\leq\al}E_r(x)\ d_mx$. For this purpose we consider the function defined on $(0,1]$
\begin{align*}
\Upsilon(\al):=\int_{x\in\R^{m+}:\ \sum_{i=1}^{m}x_i\leq\al}E_r(x)\ d_mx.
\end{align*}
Then
\begin{align*}
\begin{split}
\Upsilon(\al)
=&2m\log{r}\int_{x\in\R^{m+}:\ \sum_{i=1}^{m}x_i\leq\al}x_1\ d_mx-m\int_{x\in\R^{m+}:\ \sum_{i=1}^{m}x_i\leq\al}x_1\log{x_1}\ d_mx \\
&-\int_{x\in\R^{m+}:\ \sum_{i=1}^{m}x_i\leq\al}(1-\di\sum_{i=1}^{m}x_i)\log{(1-\di\sum_{i=1}^{m}x_i)}\ d_mx \\
=&2m\log{r}\frac{\al^{m+1}}{(m+1)!}-m\frac{\al^{m+1}}{(m+1)!}\big[\log{\al}-\di\sum_{k=2}^{m+1}\frac{1}{k}\big]
-\frac{1}{(m-1)!}\int_0^{\al}(1-x)x^{m-1}\log{(1-x)}\ dx,
\end{split}
\end{align*}
\begin{align*}
\Upsilon'(\al)=\frac{\al^{m-1}}{(m-1)!}\Big\{\big(2\log{r}+\di\sum_{k=2}^m\frac{1}{k}\big)\al-\big[\al\log{\al}+(1-\al)\log{(1-\al)}\big]\Big\},
\end{align*}
where we take $\di\sum_{k=2}^m\frac{1}{k}=0$ when $m=1$.
So if $2\log{r}+\di\sum_{k=2}^m\frac{1}{k}\geq0$, $\Upsilon'(\al)\geq0$ over $(0,1]$,
\begin{align*}
\di\max_{(0,1]}\Upsilon=\Upsilon(1),\ \Rightarrow\al_0=1.
\end{align*}
If $2\log{r}+\di\sum_{k=2}^m\frac{1}{k}<0$, let $\al_0\in(0,1)$ be the nonzero root of
$\big(2\log{r}+\di\sum_{k=2}^m\frac{1}{k}\big)\al=\al\log{\al}+(1-\al)\log{(1-\al)}$,
\begin{align}\label{intE_r}
\begin{split}
\di\max_{(0,1]}\Upsilon=&\Upsilon(\al_0)=\int_{x\in\R^{m+}:\ \sum_{i=1}^{m}x_i\leq\al_0}E_r(x)\ d_mx \\
\Big(=&\frac{1}{(m+1)!}\big[(1-\al_0^m)\log{(1-\al_0)}+\di\sum_{k=1}^{m}\frac{\al_0^k}{k}\big]\Big).
\end{split}
\end{align}
\end{proof}

\begin{rmk}
The proofs of Theorem \ref{main2} and \ref{main1} also work for a general polydisc $\di\prod_{i=1}^m D(0,r_i)$. For example, if $r=(r_1,\dots,r_m)\in[1,\infty)^m$, the function $E_r$ in Theorem \ref{main2} would be $E_r(x)=2\di\sum_{i=1}^mx_i\log{r_i}-[\di\sum_{i=1}^mx_i\log{x_i}+(1-\di\sum_{i=1}^mx_i)\log{(1-\di\sum_{i=1}^mx_i)}]$ and $\di\int_{\Si_m}E_r(x)\ d_mx=\frac{2}{(m+1)!}\di\sum_{i=1}^m\log{r_i}+\frac{1}{m!}\di\sum_{k=2}^{m+1}\frac{1}{k}$.
\end{rmk}

\section{Hole probability of $SU(2)$ polynomials}
\begin{proof}[Proof of Corollary \ref{main3}]
When $r\geq1$, $\al_0=1$. The result follows from Theorem \ref{main2}. \\
When $0<r<1$, $$x\in\R^+:\ E_r(x)=2x\log{r}-\big[x\log{x}+(1-x)\log{(1-x)}\big]\geq0\Leftrightarrow0\leq x\leq\al_0.$$
By Theorem \ref{main1},
\begin{align*}
\log{P_{0,1}(r,N)}=-N^2\int_0^{\al_0}E_r(x)\ dx+o(N^2),
\end{align*}
where the value of the integral in the corollary is due to (\ref{intE_r}) and the fact that
\begin{align*}
2\al_0\log{r}=\al_0\log{\al_0}+(1-\al_0)\log{(1-\al_0)}.
\end{align*}
\end{proof}

\begin{proof}[Proof of Theorem \ref{main4}]

Since $\partial U$ is a Jordan curve, by Carath\'{e}odory's theorem, $\phi$ can be extended to a homeomorphism $\bar{D}(0,1)\to\bar{U}$. We still use $\phi$ to denote the extension map. Thus, $\tilde{s}_N(z)=\di\sum_{k=0}^{N}c_k\sqrt{\binom{N}{k}}z^k$ is nonvanishing over $\bar{U}$ if and only if
$t_N(\omega):=\di\sum_{k=0}^{N}c_k\sqrt{\binom{N}{k}}(\phi(\omega))^k$ is nonvanishing over $\bar{D}(0,1)$,
where $t_N\in\ocal(D(0,1))\cap\ccal(\bar{D}(0,1))$.

Since
\begin{align*}
\begin{bmatrix}
t_N(0) \\
t'_N(0) \\
\vdots \\
t^{(N)}_N(0)
\end{bmatrix}
=A
\begin{bmatrix}
c_0 \\
c_1 \\
\vdots \\
c_N
\end{bmatrix},
\end{align*}
where A is an $(N+1)\times(N+1)$ lower triangular matrix with diagonal entries $\Big\{k!\sqrt{\binom{N}{k}}(\phi'(0))^k\Big\}_{0\leq k\leq N}$, $\big(t_N(0)\ \dots\ t^{(N)}_N(0)\big)^t$ is Gaussian with covariance matrix $AA^*$.
\begin{align}\label{det(AA^*)}
\det(AA^*)=\abs{\det{A}}^2=\di\prod_{k=0}^{N}\big[k!^2\binom{N}{k}\abs{\phi'(0)}^{2k}\big]\neq0
\end{align}
because $\phi$ is a biholomorphism.

We again define $\kappa=1-\sqrt{\de}$. Then if $\di\sup_{\pa D(0,\kappa)}\abs{t_N}<1$, for $0\leq k\leq N$,
\begin{align*}
\abs{t^{(k)}_N(0)}=\Abs{\frac{k!}{2\pi\sqrt{-1}}\int_{\pa D(0,\kappa)}\frac{t_N(u)}{u^{k+1}}\ du}\leq\frac{k!}{\kappa^k}.
\end{align*}
Therefore,
\begin{align*}
\begin{split}
\gamma_N(\di\sup_{\pa D(0,\kappa)}\abs{t_N}<1)
&\leq\gamma_N\Big\{\big(t_N(0),\dots,t^{(N)}_N(0)\big)\in\di\prod_{k=0}^{N}\bar{D}\big(0,\frac{k!}{\kappa^k}\big)\Big\} \\
&=\frac{1}{\pi^{N+1}\det(AA^*)}\int_{\prod_{k=0}^{N}\bar{D}\big(0,\frac{k!}{\kappa^k}\big)}\exp\{-\eta^*(AA^*)^{-1}\eta\}\ d_{2(N+1)}\eta \\
&\leq\frac{\pi^{N+1}\prod_{k=0}^{N}\big(\frac{k!}{\kappa^k}\big)^2}{\pi^{N+1}\det(AA^*)}.
\end{split}
\end{align*}
By (\ref{det(AA^*)}),
\begin{align*}
\begin{split}
\gamma_N(\di\sup_{\pa D(0,\kappa)}\abs{t_N}<1)
&\leq\frac{\prod_{k=0}^{N}\big(\frac{k!}{\kappa^k}\big)^2}{\prod_{k=0}^{N}\big[k!^2\binom{N}{k}\abs{\phi'(0)}^{2k}\big]} \\
&=\Big\{\di\prod_{k=0}^{N}\big[\binom{N}{k}(\kappa\abs{\phi'(0)})^{2k}\big]\Big\}^{-1} \\
&=\exp\{-Q_{\kappa\abs{\phi'(0)},1}(N)\} \\
&=\exp\{-(\log{\abs{\phi'(0)}}+\log{\kappa}+\half)N^2+o(N^2)\},
\end{split}
\end{align*}
where the last equality is due to Lemma \ref{lem(Q)}.

Similar as Lemma \ref{lem(logProdzeta)}, we can show that if $t_N|_{\bar{D}(0,1)}\neq0$, then outside an event of probability at most $e^{-e^N}+\exp\{-Q_{\kappa\abs{\phi'(0)},1}(N)\}=\exp\{-(\log{\abs{\phi'(0)}}+\log{\kappa}+\half)N^2+o(N^2)\}$,
\begin{align*}
\log{\di\prod_{j=0}^{N}\abs{t_N(z_j)}}\leq\frac{o(N^2)}{\de^2}+(N+1)\log{\abs{c_0}},
\end{align*}
where $z_j=\kappa e^{2\pi\sqrt{-1}\frac{j}{N+1}}$, $0\leq j\leq N$.

$\big(t_N(z_0)\ \dots\ t_N(z_N)\big)^t$ is complex Gaussian with covariance matrix
\begin{align*}
\begin{split}
\Si
&=\big(\E_N(t_N(z_j)\overline{t_N(z_j)})\big)_{0\leq i,j\leq N}=\big(\di\sum_{k=0}^{N}\binom{N}{k}(\phi(z_i))^k(\overline{\phi(z_j)})^k\big)_{0\leq i,j\leq N} \\
&=\begin{bmatrix}
\sqrt{\binom{N}{0}}&\sqrt{\binom{N}{1}}\phi(z_0)&\cdots&\sqrt{\binom{N}{N}}(\phi(z_0))^N \\
\cdots&\cdots&\cdots&\cdots \\
\sqrt{\binom{N}{0}}&\sqrt{\binom{N}{1}}\phi(z_N)&\cdots&\sqrt{\binom{N}{N}}(\phi(z_N))^N
\end{bmatrix}
\begin{bmatrix}
\sqrt{\binom{N}{0}}&\sqrt{\binom{N}{1}}\phi(z_0)&\cdots&\sqrt{\binom{N}{N}}(\phi(z_0))^N \\
\cdots&\cdots&\cdots&\cdots \\
\sqrt{\binom{N}{0}}&\sqrt{\binom{N}{1}}\phi(z_N)&\cdots&\sqrt{\binom{N}{N}}(\phi(z_N))^N
\end{bmatrix}^*
\end{split}
\end{align*}
and
\begin{align*}
\det{\Si}=\di\prod_{k=0}^{N}\binom{N}{k}\di\prod_{0\leq i<j\leq N}\abs{\phi(z_i)-\phi(z_j)}^2,
\end{align*}
\begin{align}\label{eq6}
\Rightarrow\log{\det{\Si}}=\di\sum_{k=0}^{N}\log{\binom{N}{k}}+2\di\sum_{0\leq i<j\leq N}\log{\abs{\phi(z_i)-\phi(z_j)}},
\end{align}
Next we will show that
\begin{align}\label{eq7}
2\di\sum_{0\leq i<j\leq N}\log{\abs{\phi(z_i)-\phi(z_j)}}=N^2\int_{\pa D(0,\kappa)}\int_{\pa D(0,\kappa)}\log{\abs{\phi(u_1)-\phi(u_2)}}\ d\sigma_{\kappa}(u_1)d\sigma_{\kappa}(u_2)+o_{\de}(N^2),
\end{align}
where $o_{\de}(N^2)$ denotes a lower order term depending on $\delta$.

Since
\begin{align*}
2\di\sum_{0\leq i<j\leq N}\log{\abs{\phi(z_i)-\phi(z_j)}}
=2(N+1)^2\di\sum_{0\leq i<j\leq N}\frac{1}{(N+1)^2}\log{\abs{\phi(\kappa e^{2\pi\sqrt{-1}\frac{i}{N+1}})-\phi(\kappa e^{2\pi\sqrt{-1}\frac{j}{N+1}})}}
\end{align*}
and
\begin{align*}
\begin{split}
&\int_{\pa D(0,\kappa)}\int_{\pa D(0,\kappa)}\log{\abs{\phi(u_1)-\phi(u_2)}}\ d\sigma_{\kappa}(u_1)d\sigma_{\kappa}(u_2) \\
=&\int_0^1\int_0^1\log{\abs{\phi(\kappa e^{2\pi\sqrt{-1}x})-\phi(\kappa e^{2\pi\sqrt{-1}y})}}\ dxdy \\
=&2\iint_{0\leq x\leq y\leq1}\log{\abs{\phi(\kappa e^{2\pi\sqrt{-1}x})-\phi(\kappa e^{2\pi\sqrt{-1}y})}}\ dxdy,
\end{split}
\end{align*}
it suffices to show that
\begin{align*}
\begin{split}
&\Abs{\di\sum_{0\leq i<j\leq N}\frac{1}{(N+1)^2}\log{\abs{\phi(\kappa e^{2\pi\sqrt{-1}\frac{i}{N+1}})-\phi(\kappa e^{2\pi\sqrt{-1}\frac{j}{N+1}})}}
-\iint_{0\leq x\leq y\leq1}\log{\abs{\phi(\kappa e^{2\pi\sqrt{-1}x})-\phi(\kappa e^{2\pi\sqrt{-1}y})}}\ dxdy} \\
=&o_{\de}(1).
\end{split}
\end{align*}
Since $\phi$ is a biholomorphism in $D(0,1)$, we set
\begin{align*}
\inf_{\bar{D}(0,\kappa)}\abs{\phi'}=a(\de)>0.
\end{align*}
And by Cauchy's inequality, we have
\begin{align*}
\sup_{\bar{D}(0,\kappa)}\abs{\phi'}\leq O(\de^{-1}).
\end{align*}
For each $N$, denote
\begin{align*}
\De(N)=\{(i,j)\in\Z^2:\ 0\leq i<j\leq N\},
\end{align*}
the "far from diagonal" indices
\begin{align*}
FD(N)=\left\{(i,j)\in\De(N):\
\begin{array}{cc}
\lfloor\sqrt{N+1}\rfloor+i\leq j\leq N-\lfloor\sqrt{N+1}\rfloor+i &\text{ if }0\leq i\leq\lfloor\sqrt{N+1}\rfloor \\
\lfloor\sqrt{N+1}\rfloor+i\leq j\leq N &\text{ if }\lfloor\sqrt{N+1}\rfloor<i\leq N-\lfloor\sqrt{N+1}\rfloor \\
j\in\emptyset &\text{ if }i>N-\lfloor\sqrt{N+1}\rfloor
\end{array}
\right\},
\end{align*}
\begin{align*}
\mathscr{FD}(N)=\bigcup_{(i,j)\in FD(N)}[\frac{i}{N+1},\frac{i+1}{N+1}]\times[\frac{j}{N+1},\frac{j+1}{N+1}],
\end{align*}
and the "near diagonal" indices:
\begin{align*}
D(N)=\De(N)\setminus FD(N).
\end{align*}
Then
\begin{align*}
\abs{D(N)}=O(N^{\frac{3}{2}}),
\end{align*}
and for $(i,j)\in FD(N)$,
\begin{align*}
\frac{i}{N+1}-\frac{j}{N+1}\geq(N+1)^{-\half}\ \mod1.
\end{align*}
So
\begin{align*}
\begin{split}
&\Abs{\di\sum_{0\leq i<j\leq N}\frac{1}{(N+1)^2}\log{\abs{\phi(\kappa e^{2\pi\sqrt{-1}\frac{i}{N+1}})-\phi(\kappa e^{2\pi\sqrt{-1}\frac{j}{N+1}})}}
-\iint_{0\leq x\leq y\leq1}\log{\abs{\phi(\kappa e^{2\pi\sqrt{-1}x})-\phi(\kappa e^{2\pi\sqrt{-1}y})}}\ dxdy} \\
\leq&\di\sum_{(i,j)\in D(N)}\frac{1}{(N+1)^2}\Abs{\log{\abs{\phi(\kappa e^{2\pi\sqrt{-1}\frac{i}{N+1}})-\phi(\kappa e^{2\pi\sqrt{-1}\frac{j}{N+1}})}}} \\
&+\di\sum_{(i,j)\in FD(N)}\di\int_{\frac{j}{N+1}}^{\frac{j+1}{N+1}}\di\int_{\frac{i}{N+1}}^{\frac{i+1}{N+1}}\Abs{\log{\abs{\phi(\kappa e^{2\pi\sqrt{-1}x})-\phi(\kappa e^{2\pi\sqrt{-1}y})}}-\log{\abs{\phi(\kappa e^{2\pi\sqrt{-1}\frac{i}{N+1}})-\phi(\kappa e^{2\pi\sqrt{-1}\frac{j}{N+1}})}}} dxdy \\
&+\Abs{\di\iint_{\mathscr{FD}(N)}\log{\abs{\phi(\kappa e^{2\pi\sqrt{-1}x})-\phi(\kappa e^{2\pi\sqrt{-1}y})}}\ dxdy-\di\iint_{0\leq x\leq y\leq1}\log{\abs{\phi(\kappa e^{2\pi\sqrt{-1}x})-\phi(\kappa e^{2\pi\sqrt{-1}y})}}\ dxdy} \\
=&I+II+III.
\end{split}
\end{align*}

\begin{align*}
\begin{split}
&\frac{a(\de)}{N+1}\leq\abs{\phi(\kappa e^{2\pi\sqrt{-1}\frac{i}{N+1}})-\phi(\kappa e^{2\pi\sqrt{-1}\frac{j}{N+1}})}\leq O(1)\ \ \forall(i,j)\in D(N), \\
\Rightarrow&\Abs{\log{\abs{\phi(\kappa e^{2\pi\sqrt{-1}\frac{i}{N+1}})-\phi(\kappa e^{2\pi\sqrt{-1}\frac{j}{N+1}})}}}\leq\abs{\log{a(\de)}}+\log{(N+1)}, \\
\Rightarrow&I\leq\frac{O(N^{\frac{3}{2}})}{N^2}[\abs{\log{a(\de)}}+\log{(N+1)}]=o_{\de}(1).
\end{split}
\end{align*}
Since
\begin{align*}
\di\sup_{x-y\geq(N+1)^{-\half}\mod1}\norm{\nabla\log{\abs{\phi(\kappa e^{2\pi\sqrt{-1}x})-\phi(\kappa e^{2\pi\sqrt{-1}y})}}}\leq\frac{O(\de^{-1})}{a(\de)(N+1)^{-\half}}=\frac{O(N^{\half})}{\de a(\de)},
\end{align*}
\begin{align*}
\begin{split}
\Rightarrow II\leq&\frac{N^2}{(N+1)^2}\di\sup_{x-y\geq(N+1)^{-\half}\mod1}\norm{\nabla\log{\abs{\phi(\kappa e^{2\pi\sqrt{-1}x})-\phi(\kappa e^{2\pi\sqrt{-1}y})}}}O(N^{-1}) \\
\leq&\frac{O(N^{-\half})}{\de a(\de)}=o_{\de}(1).
\end{split}
\end{align*}
By a similar argument as Lemma \ref{lem(detW2)}, we have
\begin{align*}
\di\lim_{N\to\infty}\vol_{\R^2}\left(\mathscr{FD}(N)\triangle\{(x,y)\in\R^2:\ 0\leq x\leq y\leq1\}\right)=0.
\end{align*}
Furthermore, (\ref{eq8}) and (\ref{eq9}) below indicate that the function $\log{\abs{\phi(\kappa e^{2\pi\sqrt{-1}x})-\phi(\kappa e^{2\pi\sqrt{-1}y})}}$ is $L^1$ over $[0,1]^2$,
\begin{align*}
\Rightarrow III\leq o_{\de}(1).
\end{align*}
Thus, we have proved (\ref{eq7}).

For $u_1,\ u_2\in D(0,1)$, define:
\begin{align*}
\psi(u_1,u_2)=
\begin{cases}
\frac{\phi(u_1)-\phi(u_2)}{u_1-u_2}&\text{ if }u_1\neq u_2, \\
\phi'(u_1)&\text{ if }u_1=u_2.
\end{cases}
\end{align*}
Then $\psi$ is continuous and nonzero in $D(0,1)\times D(0,1)$. Moreover, by removable singularity theorem, $\psi$ is holomorphic in $u_1$ as well as $u_2$. Therefore, $\log{\abs{\psi}}$ is pluriharmonic in $D(0,1)\times D(0,1)$. By mean value equality,
\begin{align}\label{eq8}
\begin{split}
&\int_{\pa D(0,\kappa)}\int_{\pa D(0,\kappa)}\log{\abs{\phi(u_1)-\phi(u_2)}}\ d\sigma_{\kappa}(u_1)d\sigma_{\kappa}(u_2) \\
=&\int_{\pa D(0,\kappa)}\int_{\pa D(0,\kappa)}\log{\abs{\psi(u_1,u_2)}}\ d\sigma_{\kappa}(u_1)d\sigma_{\kappa}(u_2)
+\int_{\pa D(0,\kappa)}\int_{\pa D(0,\kappa)}\log{\abs{u_1-u_2}}\ d\sigma_{\kappa}(u_1)d\sigma_{\kappa}(u_2) \\
=&\log{\abs{\psi(0,0)}}+\log{\kappa}+\int_{\pa D(0,1)}\int_{\pa D(0,1)}\log{\abs{u_1-u_2}}\ d\sigma_1(u_1)d\sigma_1(u_2) \\
=&\log{\abs{\phi'(0)}}+\log{\kappa}+\int_{\pa D(0,1)}\int_{\pa D(0,1)}\log{\abs{u_1-u_2}}\ d\sigma_1(u_1)d\sigma_1(u_2),
\end{split}
\end{align}
\begin{align}\label{eq9}
\begin{split}
&\int_{\pa D(0,1)}\int_{\pa D(0,1)}\log{\abs{u_1-u_2}}\ d\sigma_1(u_1)d\sigma_1(u_2) \\
=&\int_0^1\int_0^1\log{\abs{e^{2\pi\sqrt{-1}x}-e^{2\pi\sqrt{-1}y}}}\ dxdy \\
=&\int_0^1\log{\abs{1-e^{2\pi\sqrt{-1}x}}}\ dx \\
=&\int_{\pa D(0,1)}\log{\abs{1-z}}\ d\sigma_1(z) \\
=&0,
\end{split}
\end{align}
where the last equality is due to Lebesgue's dominated convergence theorem.

(\ref{eq6})$\sim$(\ref{eq9}) show that
\begin{align*}
\begin{split}
\log{\det{\Si}}
&=\di\sum_{k=0}^{N}\log{\binom{N}{k}}+(\log{\abs{\phi'(0)}}+\log{\kappa})N^2+o_{\de}(N^2) \\
&=(\log{\abs{\phi'(0)}}+\log{\kappa}+\half)N^2+o_{\de}(N^2).
\end{split}
\end{align*}
The remaining part is similar to Section 4.
\end{proof}

\begin{rmk}
For $U=D(0,r)$, $\phi$ would be a rotation composed with a scaling by $r$. So $\abs{\phi'(0)}=r$. Thus the upper bound in Theorem \ref{main4} is $-(\log{r}+\half)N^2+o(N^2)$, which agrees with Corollary \ref{main3} in case of $r\geq1$.
\end{rmk}

\section{Generalized hole probabilities of $SU(2)$ polynomials}
If $n(r,N)$ denotes the number of zeros of $\tilde{s}_N(z)$ in $\bar{D}(0,r)$ counting multiplicity, then the hole probability $P_{0,1}(r,N)$ is just the first term of the sequence of the probabilities
\begin{align*}
P_{k,1}(r,N)=\gamma_N\{n(r,N) \leq k\},\ k\geq0.
\end{align*}
We call $P_{k,1}(r,N)$ a generalized hole probability because compared with the large degree or total number of zeros in $\C$ of the polynomial $\tilde{s}_N$, any finite number $k$ is negligible. It is a status of almost having no zero in $D(0,r)$. And by Theorem \ref{main5}, it turns out that the generalized hole probabilities are numerically almost equal to the regular one.
\begin{proof}[Proof of Theorem \ref{main5}]
(\ref{ineq(sup)}) implies that $\forall\ \eta>0$,
\begin{align}\label{upbd}
\gamma_N\big\{\int_{\pa D(0,r)}\log{\abs{\tilde{s}_N(u)}}\ d\sigma_r(u)>\frac{N}{2}\log{(1+r^2)}+\eta N\big\}\leq e^{-e^{\eta N}}\text{ for }N\gg1.
\end{align}
We follow the notations in Section 5 except this time $m=1$ and we take the number of partitions $p=1$. The corresponding statement of Lemma \ref{lem(logProdzeta'')} is
\begin{align*}
\gamma_N\big\{
\log{\di\prod_{j=0}^{\lfloor\al_0N\rfloor}\abs{\zeta_j}}>\frac{o(N^2)}{\de^2}+(\lfloor\al_0N\rfloor+1)\int_{\pa D(0,r)}\log{\abs{\tilde{s}_N(u)}}\ d\sigma_r(u)\big\}
\leq e^{-e^N}+e^{-R_{\kappa r,1}(N)},
\end{align*}
where $\zeta_j=\tilde{s}_N(\kappa re^{2\pi\sqrt{-1}\frac{j}{\lfloor\al_0N\rfloor+1}})$, $0\leq j\leq\lfloor\al_0N\rfloor$.
Here we do not need to assume $0\not\in\tilde{s}_N\big({\bar{D}(0,r)}\big)$ as in Lemma \ref{lem(logProdzeta'')}: the counterpart of $II$ in (\ref{ineq2}) is
\begin{align*}
II=(\lfloor\al_0N\rfloor+1)\int_{\pa D(0,r)}\log{\abs{\tilde{s}_N(u)}}\int_HP_r(\kappa re^{2\pi\sqrt{-1}x},u)\ dxd\sigma_r(u).
\end{align*}
Since $m=1$ and $p=1$, $H=[0,1]\subset\R$,
\begin{align*}
\begin{split}
II
&=(\lfloor\al_0N\rfloor+1)\int_{\pa D(0,r)}\log{\abs{\tilde{s}_N(u)}}\int_0^1P_r(\kappa re^{2\pi\sqrt{-1}x},u)\ dxd\sigma_r(u) \\
&=(\lfloor\al_0N\rfloor+1)\int_{\pa D(0,r)}\log{\abs{\tilde{s}_N(u)}}\ d\sigma_r(u).
\end{split}
\end{align*}
Therefore, $\forall\ \eta>0$ small,
\begin{align}\label{ineq12}
\begin{split}
&\gamma_N\big\{\int_{\pa D(0,r)}\log{\abs{\tilde{s}_N(u)}}\ d\sigma_r(u)\leq\frac{N}{2}\log{(1+r^2)}-\eta N\big\} \\
\leq&e^{-e^N}+e^{-R_{\kappa r,1}(N)}
+\gamma_N\Big\{\di\prod_{j=0}^{\lfloor\al_0N\rfloor}\abs{\zeta_j}\leq\exp{\big\{\frac{o(N^2)}{\de^2}+(\lfloor\al_0N\rfloor+1)[\frac{N}{2}\log{(1+r^2)}-\eta N]\big\}}\Big\}.
\end{split}
\end{align}
Following the steps (\ref{ineq11})$\sim$(\ref{ineq(F)}), we can show that
\begin{align*}
\begin{split}
&\log{\gamma_N\big\{\int_{\pa D(0,r)}\log{\abs{\tilde{s}_N(u)}}\ d\sigma_r(u)\leq\frac{N}{2}\log{(1+r^2)}-\eta N\big\}} \\
\leq&N(\lfloor\al_0N\rfloor+1)[\log(1+r^2)-2\eta]-Q_{\kappa r,1,\al_0}(N)-2\beta_1\al_0^2N^2+\frac{o(N^2)}{\de^2}.
\end{split}
\end{align*}
\begin{align*}
Q_{\kappa r,1,\al_0}(N)
\sim N^2\int_0^{\al_0}E_r(x)\ dx=\half\al_0[2\log{\kappa r}+1-\log{\al_0}]N^2,
\end{align*}
\begin{align*}
\begin{split}
\beta_1
&=\int_0^1x\log{[2\sin(\pi x)]}\ dx \\
&=\int_0^1(x-\half)\log{[2\sin(\pi x)]}\ dx+\half\int_0^1\log{[2\sin(\pi x)]}\ dx \\
&=\int_{-\half}^{\half}x\log{[2\sin\pi(x+\half)]}\ dx+\half\int_0^1\log{[2\sin(\pi x)]}\ dx \\
&=\int_{-\half}^{\half}x\log{[2\cos(\pi x)]}\ dx+\half\int_0^1\log{[2\sin(\pi x)]}\ dx,
\end{split}
\end{align*}
as $\di\int_{-\half}^0x\log{[2\cos(\pi x)]}\ dx$ and $\di\int_0^{\half}x\log{[2\cos(\pi x)]}\ dx$ both converge and $x\log{[2\cos(\pi x)]}$ is odd,
\begin{align*}
\beta_1=\half\int_0^1\log{[2\sin(\pi x)]}\ dx=\half\int_{\pa D(0,1)}\log{\abs{1-z}}\ d\sigma_1(z),
\end{align*}
which equals $0$ as in (\ref{eq9}).
Thus
\begin{align}\label{ineq13}
\begin{split}
&\log{\gamma_N\big\{\int_{\pa D(0,r)}\log{\abs{\tilde{s}_N(u)}}\ d\sigma_r(u)\leq\frac{N}{2}\log{(1+r^2)}-\eta N\big\}} \\
\leq&-\half\al_0[1+2\log{(\kappa r)}-\log{\al_0}-2\log{(1+r^2)}+4\eta]N^2+\frac{o(N^2)}{\de^2}.
\end{split}
\end{align}
On the other hand,
\begin{align}\label{ineq14}
R_{\kappa r,1}(N)\sim N^2\int_{E_{\kappa r}(x)\geq0}E_{\kappa r}(x)\ dx.
\end{align}
Combine (\ref{ineq12})$\sim$(\ref{ineq14}), and let $\de\to0+$, we get
\begin{align}\label{ineq15}
\begin{split}
&\log{\gamma_N\big\{\int_{\pa D(0,r)}\log{\abs{\tilde{s}_N(u)}}\ d\sigma_r(u)\leq\frac{N}{2}\log{(1+r^2)}-\eta N\big\}} \\
\leq&-\min{\big\{\half\al_0[1+2\log{r}-\log{\al_0}-2\log{(1+r^2)}+4\eta],\half\al_0[1+2\log{r}-\log{\al_0}]\big\}}N^2+o(N^2) \\
=&-\half\al_0[1+2\log{r}-\log{\al_0}-2\log{(1+r^2)}+4\eta]N^2+o(N^2),
\end{split}
\end{align}
for $0<\eta<\half\log{(1+r^2)}$.
Since
\begin{align*}
\int_{E_r(x)\geq0}E_r(x)\ dx=\half\al_0[1+2\log{r}-\log{\al_0}]>0\Rightarrow 1+2\log{r}-\log{\al_0}>0,
\end{align*}
we can choose $0<\eta<\half\log{(1+r^2)}$ close to $\half\log{(1+r^2)}$ such that
\begin{align*}
1+2\log{r}-\log{\al_0}-2\log{(1+r^2)}+4\eta>0.
\end{align*}
Therefore (\ref{ineq15}) makes sense.
Denote
\begin{align*}
F_\eta(r)=\half\al_0[1+2\log{r}-\log{\al_0}-2\log{(1+r^2)}+4\eta],
\end{align*}
so we have
\begin{align}\label{lwbd}
\gamma_N\big\{\int_{\pa D(0,r)}\log{\abs{\tilde{s}_N(u)}}\ d\sigma_r(u)\leq\frac{N}{2}\log{(1+r^2)}-\eta N\big\}\leq e^{-F_\eta(r)N^2+o(N^2)},\
0<\eta<\half\log{(1+r^2)}.
\end{align}

Let $\rho>1$ to be determined. By discarding a null set, we may assume $\tilde{s}_N(0)\neq0$, $0\not\in\tilde{s}_N\big(\pa D(0,r)\big)$ and $0\not\in\tilde{s}_N\big(\pa D(0,\rho^{-1}r)\big)$.

So by Jensen's formula, almost surely,
\begin{align}
\int_{\pa D(0,r)}\log{\abs{\tilde{s}_N(u)}}\ d\sigma_r(u)&=\log{\abs{c_0}}+\int_{0}^{r}\frac{n(t,N)}{t}\ dt, \label{eq10}\\
\int_{\pa D(0,\rho^{-1}r)}\log{\abs{\tilde{s}_N(u)}}\ d\sigma_{\rho^{-1}r}(u)&=\log{\abs{c_0}}+\int_{0}^{\rho^{-1}r}\frac{n(t,N)}{t}\ dt. \label{eq11}
\end{align}
Since $n(r,N)$ is increasing with respect to $r$,
\begin{align*}
(\ref{eq10})\sim(\ref{eq11})\Rightarrow
&\int_{\pa D(0,r)}\log{\abs{\tilde{s}_N(u)}}\ d\sigma_r(u)-\int_{\pa D(0,\rho^{-1}r)}\log{\abs{\tilde{s}_N(u)}}\ d\sigma_{\rho^{-1}r}(u) \\
=&\int_{\rho^{-1}r}^{r}\frac{n(t,N)}{t}\ dt\leq n(r,N)\log{\rho},
\end{align*}
\begin{align}\label{lwbd2}
\Rightarrow\ n(r,N)\geq\frac{1}{\log{\rho}}\big[\int_{\pa D(0,r)}\log{\abs{\tilde{s}_N(u)}}\ d\sigma_r(u)
-\int_{\pa D(0,\rho^{-1}r)}\log{\abs{\tilde{s}_N(u)}}\ d\sigma_{\rho^{-1}r}(u)\big].
\end{align}
(\ref{upbd}) $\Rightarrow$ For $\eta_1>0$, outside an event of probability at most $e^{-e^{\eta_1 N}}$,
\begin{align}\label{est7}
\int_{\pa D(0,\rho^{-1}r)}\log{\abs{\tilde{s}_N(u)}}\ d\sigma_{\rho^{-1}r}(u)\leq \frac{N}{2}\log{(1+\rho^{-2}r^2)}+\eta_1 N,
\end{align}
(\ref{lwbd}) $\Rightarrow$ For $0<\eta_2<\half\log{(1+r^2)}$, outside an event of probability at most $e^{-F_{\eta_2}(r)N^2+o(N^2)}$,
\begin{align}\label{est8}
\int_{\pa D(0,r)}\log{\abs{\tilde{s}_N(u)}}\ d\sigma_r(u)\geq \frac{N}{2}\log{(1+r^2)}-\eta_2 N.
\end{align}
(\ref{lwbd2})$\sim$(\ref{est8}) $\Rightarrow$ outside an event of probability at most $e^{-e^{\eta_1 N}}+e^{-F_{\eta_2}(r)N^2+o(N^2)}$,
\begin{align*}
n(r,N)\geq\frac{N}{\log{\rho}}[\half\log{(1+r^2)}-\half\log{(1+\rho^{-2}r^2)}-(\eta_1+\eta_2)].
\end{align*}
\begin{align*}
\Rightarrow\gamma_N\big\{n(r,N)<\frac{N}{\log{\rho}}[\half\log{(1+r^2)}-\half\log{(1+\rho^{-2}r^2)}-(\eta_1+\eta_2)]\big\}\leq e^{-e^{\eta_1N}}+e^{-F_{\eta_2}(r)N^2+o(N^2)},
\end{align*}
where the right hand side is independent of $\rho$. We need to choose proper $\rho$, $\eta_1$ and $\eta_2$.

$\forall\ \tau>0$, we set
\begin{align*}
\frac{1}{\log{\rho}}[\half\log{(1+r^2)}-\half\log{(1+\rho^{-2}r^2)}-(\eta_1+\eta_2)]=\tau,
\end{align*}
\begin{align*}
\eta_1+\eta_2=\eta_{\tau}(\rho):=\half\log{(1+r^2)}-\half\log{(1+\rho^{-2}r^2)}-\tau\log{\rho}.
\end{align*}
If $\tau>0$ is small enough, $\rho_0(\tau):=\sqrt{\frac{1-\tau}{\tau}}r>1$,
\begin{align*}
\eta'_{\tau}(\rho)=\frac{\rho^{-3}r^2}{1+\rho^{-2}r^2}-\frac{\tau}{\rho}=\frac{(1-\tau)r^2-\tau\rho^2}{\rho(\rho^2+r^2)}
\begin{cases}
>0 &\text{ when }1<\rho<\rho_0, \\
=0 &\text{ when }\rho=\rho_0, \\
<0 &\text{ when }\rho>\rho_0.
\end{cases}
\end{align*}
\begin{align*}
\Rightarrow(\eta_1+\eta_2)_{\max}
&=\eta_{\tau}(\rho_0(\tau)) \\
&=\half\log{(1+r^2)}-\half\log{(1+\frac{\tau}{1-\tau})}-\tau[\half\log{(1-\tau)}-\half\log{\tau}+\log{r}] \\
&=\half\log{(1+r^2)}+\half\log{(1-\tau)}-\frac{\tau}{2}\log{(1-\tau)}+\frac{\tau}{2}\log{\tau}-\tau\log{r} \\
&=\half\log{(1+r^2)}+\half[\tau\log{\tau}+(1-\tau)\log{(1-\tau)}-2\tau\log{r}].
\end{align*}
For a fixed $r>0$, we can choose smaller $\tau>0$ if necessary so that $$-\half\log{(1+r^2)}<\tau\log{\tau}+(1-\tau)\log{(1-\tau)}-2\tau\log{r}<0.$$ This is possible since $$\tau\log{\tau}+(1-\tau)\log{(1-\tau)}-2\tau\log{r}<0\text{ if }0<\tau<\al_0$$ and $$\di\lim_{\tau\to0+}[\tau\log{\tau}+(1-\tau)\log{(1-\tau)}-2\tau\log{r}]=0.$$ Thus for such $\tau$ and the corresponding $\rho_0=\rho_0(\tau)$, $$\frac{1}{4}\log{(1+r^2)}<\eta_1+\eta_2=\eta_{\tau}(\rho_0)<\half\log{(1+r^2)}.$$
In this case, $\forall\ 0<\eta_1<\frac{1}{4}\log{(1+r^2)}$,
\begin{align*}
0<\eta_2=\half\log{(1+r^2)}+\half[\tau\log{\tau}+(1-\tau)\log{(1-\tau)}-2\tau\log{r}]-\eta_1<\half\log{(1+r^2)},
\end{align*}
\begin{align*}
\gamma_N\{n(r,N)<\tau N\}
&=\gamma_N\big\{n(r,N)<\frac{N}{\log{\rho_0}}[\half\log{(1+r^2)}-\half\log{(1+\rho_0^{-2}r^2)}-(\eta_1+\eta_2)]\big\} \\
&\leq e^{-e^{\eta_1N}}+e^{-F_{\eta_2}(r)N^2+o(N^2)}.
\end{align*}
$\forall\ k\geq0$, for $N$ large enough, $k<\tau N$,
\begin{align*}
&\exp{\{-\half\al_0(1+2\log{r}-\log{\al_0})N^2+o(N^2)\}}=P_{0,1}(r,N)\leq P_{k,1}(r,N)\leq\gamma_N\{n(r,N)<\tau N\} \\
\leq&e^{-e^{\eta_1N}}+\exp\bigg\{-\half\al_0\Big\{\big(1+2\log{r}-\log{\al_0}\big)+2\big[\tau\log{\tau}+(1-\tau)\log{(1-\tau)}-2\tau\log{r}\big]-4\eta_1\Big\}N^2+o(N^2)\bigg\}.
\end{align*}
Therefore,
\begin{align*}
-\half\al_0(1+2\log{r}-\log{\al_0})\leq&\di\liminf_{N\to\infty}\frac{\log{P_{k,1}(r,N)}}{N^2}\leq\di\limsup_{N\to\infty}\frac{\log{P_{k,1}(r,N)}}{N^2} \\
\leq&-\half\al_0\Big\{\big(1+2\log{r}-\log{\al_0}\big)+2\big[\tau\log{\tau}+(1-\tau)\log{(1-\tau)}-2\tau\log{r}\big]-4\eta_1\Big\}.
\end{align*}
Let $\eta_1\to0+$ and then $\tau\to0+$,
\begin{align*}
\Rightarrow\ \di\lim_{N\to\infty}\frac{\log{P_{k,1}(r,N)}}{N^2}=-\half\al_0(1+2\log{r}-\log{\al_0})\
\Leftrightarrow\ \log{P_{k,1}(r,N)}\sim-\half\al_0(1+2\log{r}-\log{\al_0})N^2.
\end{align*}
\end{proof}

\section{Appendix}

We now prove the following lemma:
\begin{lem}
The coefficient of $g_{m,N}(\xi)$ in $\det{W_{m,N}(\xi)}$ equals $1$.
\end{lem}
\begin{proof}
Let $\scal_{m,N}$ be the set of bijections from $\Ga_{m,N}$ to $\La_{m,N}$ and $\forall\ \si\in\scal_{m,N}$, $J\in\Ga_{m,N}$, write $\si(J)=(\si_1(J),\dots,\si_m(J))$. Then
\begin{align*}
\det{W_{m,N}(\xi)}=\sum_{\si\in\scal_{m,N}}sgn({\si})\prod_{J\in\Ga_{m,N}}\xi_{J}^{\si(J)}=\sum_{\si\in\scal_{m,N}}sgn({\si})\prod_{J\in\Ga_{m,N}}\xi_{1,j_1}^{\si_1(J)}\cdots\xi_{m,j_m}^{\si_m(J)}.
\end{align*}
To find those $\si\in\scal_{m,N}$ ending up with $g_{m,N}(\xi)$, it is equivalent to find $\si$ satisfying $\forall\ 1\leq i\leq m$,
\begin{align}\label{eq3}
\sum_{J\in\Ga_{m,N}^{i,k}}\si_i(J)=
\begin{cases}
\binom{k+i-1}{i}\binom{N-k+m-i}{m-i}\ &1\leq k\leq N, \\
0\ &k=0,
\end{cases}
\end{align}
where the set $\Ga_{m,N}^{i,k}$ is defined in (\ref{Ga(m,N,i,k)}). We are going to prove by induction that
\begin{align}\label{si}
\si(J)=(j_1,j_2-j_1,\dots,j_m-j_{m-1})\text{ for all }J\in\Ga_{m,N}.
\end{align}
First of all, similar to $\Ga_{m,N}^{i,k}$, we introduce
\begin{align*}
\La_{m,N}^{i,k}=\{(k_1,\dots,k_m)\in\La_{m,N}:k_1+\dots+k_i=k\},
\end{align*}
\begin{align*}
\La_{m,N}=\di\bigsqcup_{k=0}^{N}\La_{m,N}^{i,k},\forall\ 1\leq i\leq m\text{ and }
\abs{\La_{m,N}^{i,k}}=\binom{k+i-1}{i-1}\binom{N-k+m-i}{m-i}=\abs{\Ga_{m,N}^{i,k}}.
\end{align*}
When $i=1$, (\ref{eq3}) shows
\begin{align}\label{eq4}
\sum_{J\in\Ga_{m,N}^{1,k}}\si_1(J)=k\binom{N-k+m-1}{m-1},\ 0\leq k\leq N,
\end{align}
where the number of terms in the summation on the left is $\abs{\Ga_{m,N}^{1,k}}=\binom{N-k+m-1}{m-1}=\abs{\La_{m,N}^{1,k}}$, $\forall\ 0\leq k\leq N$. Then
\begin{align*}
\begin{split}
&k=0\text{ in }(\ref{eq4})\ \Rightarrow\ \si(\Ga_{m,N}^{1,0})=\La_{m,N}^{1,0}\ \Rightarrow\ \si(\bigsqcup_{k=1}^{N}\Ga_{m,N}^{1,k})=\bigsqcup_{k=1}^{N}\La_{m,N}^{1,k}, \\
&k=1\text{ in }(\ref{eq4})\ \Rightarrow\ \si(\Ga_{m,N}^{1,1})=\La_{m,N}^{1,1}\ \Rightarrow\ \si(\bigsqcup_{k=2}^{N}\Ga_{m,N}^{1,k})=\bigsqcup_{k=2}^{N}\La_{m,N}^{1,k}, \\
&\dots \\
&k=N\text{ in }(\ref{eq4})\ \Rightarrow\ \si(\Ga_{m,N}^{1,N})=\La_{m,N}^{1,N},
\end{split}
\end{align*}
\begin{align*}
\Rightarrow \si_1(J)=j_1,\ \forall\ J\in\Ga_{m,N}.
\end{align*}
Now assume for some $1\leq i\leq m-1$, $(\si_1+\dots+\si_i)(J)=j_i$, $\forall\ J\in\Ga_{m,N}$. Then $\forall\ 1\leq k\leq N$,
\begin{align*}
\begin{split}
\sum_{J\in\Ga_{m,N}^{i+1,k}}(\si_1+\dots+\si_{i+1})(J)
&=\sum_{J\in\Ga_{m,N}^{i+1,k}}[j_i+\si_{i+1}(J)] \\
&=\sum_{j=0}^{k}j\abs{\Ga_{m,N}^{i,j}\cap\Ga_{m,N}^{i+1,k}}+\binom{k+i}{i+1}\binom{N-k+m-i-1}{m-i-1} \\
&=\sum_{j=0}^{k}j\binom{j+i-1}{i-1}\binom{N-k+m-i-1}{m-i-1}+\binom{k+i}{i+1}\binom{N-k+m-i-1}{m-i-1} \\
&=k\binom{k+i}{i}\binom{N-k+m-i-1}{m-i-1},
\end{split}
\end{align*}
where the second term in the second equality comes from (\ref{eq3}).
And for $k=0$,
\begin{align*}
\sum_{J\in\Ga_{m,N}^{i+1,0}}(\si_1+\dots+\si_{i+1})(J)=\sum_{J\in\Ga_{m,N}^{i+1,0}}[j_i+\si_{i+1}(J)]=0.
\end{align*}
So $\forall\ 0\leq k\leq N$,
\begin{align}\label{eq5}
\sum_{J\in\Ga_{m,N}^{i+1,k}}(\si_1+\dots+\si_{i+1})(J)=k\binom{k+i}{i}\binom{N-k+m-i-1}{m-i-1},
\end{align}
where the number of terms in the summation on the left is $\abs{\Ga_{m,N}^{i+1,k}}=\binom{k+i}{i}\binom{N-k+m-i-1}{m-i-1}=\abs{\La_{m,N}^{i+1,k}}$, $\forall\ 0\leq k\leq N$.
\begin{align*}
\begin{split}
&k=0\text{ in }(\ref{eq5})\ \Rightarrow\ \si(\Ga_{m,N}^{i+1,0})=\La_{m,N}^{i+1,0}\ \Rightarrow\ \si(\bigsqcup_{k=1}^{N}\Ga_{m,N}^{i+1,k})=\bigsqcup_{k=1}^{N}\La_{m,N}^{i+1,k}, \\
&k=1\text{ in }(\ref{eq5})\ \Rightarrow\ \si(\Ga_{m,N}^{i+1,1})=\La_{m,N}^{i+1,1}\ \Rightarrow\ \si(\bigsqcup_{k=2}^{N}\Ga_{m,N}^{i+1,k})=\bigsqcup_{k=2}^{N}\La_{m,N}^{i+1,k}, \\
&\dots \\
&k=N\text{ in }(\ref{eq5})\ \Rightarrow\ \si(\Ga_{m,N}^{i+1,N})=\La_{m,N}^{i+1,N},
\end{split}
\end{align*}
\begin{align*}
\Rightarrow (\si_1+\dots+\si_{i+1})(J)=j_{i+1},\ \forall\ J\in\Ga_{m,N}.
\end{align*}
Thus, (\ref{si}) is proved. And it is trivial to check that the $\si$ defined in (\ref{si}) satisfies all the equations in (\ref{eq3}). This means that there is only one $\si\in\scal_{m,N}$ that ends up with $g_{m,N}(\xi)$, and it turns out to be order preserving. Therefore,
\begin{align*}
\det{W_{m,N}(\xi)}=g_{m,N}(\xi)+\dots
\end{align*}
\end{proof}

\end{document}